\documentclass[a4paper, 11pt]{article}

\usepackage{fullpage}

\usepackage{amsmath}
\usepackage{amsthm}
\usepackage{amsfonts}
\usepackage{tikz-cd}
\usepackage[utf8x]{inputenc}
\usepackage{enumerate}
\usepackage[english]{babel}
\usepackage{amssymb}
\usepackage{esint}
\usepackage{hyperref}
\usepackage{bm}
\usepackage{bbm}

\theoremstyle{plain}
\newtheorem{theorem}{Theorem}[section]
\newtheorem{lemma}[theorem]{Lemma}

\newtheorem{proposition}[theorem]{Proposition}
\theoremstyle{definition}
\newtheorem{definition}[theorem]{Definition}
\theoremstyle{remark}

\newtheorem{remark}[theorem]{Remark}

\numberwithin{equation}{section} 

\DeclareMathOperator{\divv}{div}
\DeclareMathOperator*{\esssup}{ess\:sup}
\DeclareMathOperator{\trace}{Tr}

\title{Semiflow selection to models of general\\ compressible viscous fluids}
\author{Danica Basari\'{c}}
\date{}

\begin{document}
	\maketitle
	
	\begin{center}
		Technische Universit\"{a}t Berlin \\
		Institute f\"{u}r Mathematik, Stra{\ss}e des 17 Juni 136, 10623 Berlin, Germany 
	\end{center}
	
	\begin{center}
		E-mail address: basaric@math.tu-berlin.de
	\end{center}

	\begin{abstract}
			We prove the existence of a semiflow selection with range the space of c\`{a}gl\`{a}d, i.e. left--continuous and having right--hand limits functions defined on $[0,\infty)$ and taking values in a Hilbert space. Afterwards, we apply this abstract result to the system arising from a compressible viscous fluid with a barotropic pressure of the type $a\varrho^{\gamma}$, $\gamma \geq 1$, with a viscous stress tensor being a nonlinear function of the symmetric velocity gradient.
	\end{abstract}

	\section{Introduction}
	
	First developed by Krylov \cite{Kry} and later adapted by Flandoli and Romito \cite{FlaRom}, Breit, Feireisl and Hofmanov\'{a} \cite{BreFeiHofm} in the context of the Navier--Stokes system, the semiflow selection is an important stochastic tool when studying systems that lack uniqueness: it allows to identify a solution satisfying at least the semigroup property. 
	
	Inspired by the deterministic adaptation of Cardona and Kapitanski \cite{CarKap}, in the first part of this work we will prove the existence of a semiflow selection in an abstract setting. More precisely, denoting with $H$ a Hilbert space and with $\mathcal{T}=\mathfrak{D}([0,\infty); H)$ the Skorokhod space of  c\`{a}gl\`{a}d functions defined on $[0,\infty)$ and taking values in $H$, we will show the existence of a Borel measurable map $\mathfrak{u}: D\subseteq H \rightarrow \mathcal{T}$ such that for any $x\in D$ and any $t_1, t_2 \geq 0$ 
	\begin{equation*}
		\mathfrak{u}(x)(t_1 + t_2) = \mathfrak{u} \left[ \mathfrak{u} (x)(t_1)\right] (t_2).
	\end{equation*}
	
	One could think that a more natural choice for $\mathcal{T}$ would be the space $C([0,\infty); H)$ of continuous functions, as in \cite{CarKap}. However, this option can be too strong: if we want to apply this abstract setting to the typical systems arising from fluid dynamics, where $[0,\infty)$ is the set of times and $\mathcal{T}$ represents the trajectory space, then it is difficult to ensure the energy of the system to be continuous, since it is at most a non--increasing quantity with possible jumps. For the aforementioned reason, in the context of the compressible Euler system, Breit, Feireisl and Hofmanov\'{a} \cite{BreFeiHofm1}, \cite{BreFeiHof} considered the energy in the $L^1$--space. But the choice $\mathcal{T}=L^1([0,\infty); H)$ is still not optimal as it is better to work with a space whose elements are well--defined at any point.
	
	In the second part of this work, we will apply the abstract machinery previously achieved to a general model of compressible viscous fluids, described by the following pair of equations:
	\begin{equation} \label{compressible viscous fluid}
		\begin{aligned}
			\partial_t \varrho + \divv_x (\varrho \textbf{u})&=0, \\
			\partial_t(\varrho \textbf{u}) + \divv_x(\varrho \textbf{u} \otimes \textbf{u})+ \nabla_x p &= \divv_x \mathbb{S},
		\end{aligned}
	\end{equation}
	where the unknown variables are the density $\varrho$ and the velocity $\textbf{u}$. In this context, the viscous stress tensor $\mathbb{S}$ is connected to the symmetric velocity gradient $\mathbb{D} \textbf{u}$ through the relation
	\begin{equation*}
		\mathbb{S} : \mathbb{D} \textbf{u}= F(\mathbb{D} \textbf{u}) +F^*(\mathbb{S}),
	\end{equation*}
	where $F$ is a proper convex l.s.c. (lower semi--continuous) function and $F^*$ denotes its conjugate; moreover, the barotropic pressure will be of the type $p(\varrho)= a\varrho^{\gamma}$, $\gamma \geq 1$. We will deal with the concept of \textit{dissipative solutions}, i.e. solutions that satisfy our problem in the weak sense but with an extra defect term in the balance of momentum, arising from possible concentrations and/or oscillations in the convective and pressure terms; for more details see Definition \ref{dissipative solution}. It is interesting to note that for $\gamma=1$, the defect in the momentum equation vanishes and the latter is satisfied in the sense of distributions. Thus our approach represents an alternative to the ``standard" measure--valued framework applied in this context by Matu\v{s}\accent23u-Ne\v{c}asov\'{a} and Novotn\'{y} \cite{MatNov}.
	
	Introducing the set--valued mapping $\mathcal{U}: D \subseteq H \rightarrow \mathcal{P}(\mathcal{T})$ that associated to every initial data $x\in D$ the family of dissipative solutions $\mathcal{U}(x)$ arising from $x \in D$, the key point in order to get the existence of semiflow selection will be to show that $\mathcal{U}$ satisfies five properties: \textit{non--emptiness}, \textit{compactness}, \textit{Borel--measurability}, \textit{shift invariance} and \textit{continuation}.
	
	As we will see, in order to verify the validity of the above mentioned properties, there are two main difficulties we have to overcome: the \textit{existence} of dissipative solutions and the \textit{weak sequential stability} of the family of dissipative solutions, arising from a fixed initial data. Luckily, the first problem was recently solved by Abbatiello, Feireisl and Novotn\'{y} in \cite{AbbFeiNov} for $\gamma>1$; following the same strategy, the case $\gamma=1$ can be done as well replacing Lemma 8.1 in \cite{AbbFeiNov} with Lemma \ref{auxiliary lemma} below. Conversely, larger part of this work will be dedicated to solve the second issue.
	
	The paper is organized as follows.
	\begin{itemize}
		\item In Section \ref{Skorokhod space} we will define the Skorokhod space of c\`{a}gl\`{a}d functions, constructing a proper metric on it and giving a characterization of convergence.
		\item In Section \ref{Semiflow selection} we will prove the existence of a semiflow  selection in the abstract setting, cf. Theorem \ref{existence semiflow selection}.
		\item In Section \ref{Compressible viscous fluids} we will focus on system \eqref{compressible viscous fluid}, first clarifying the concept of dissipative solution (Section \ref{Dissipative solution}), fixing a proper setting (Section \ref{Set--up}) and finally proving in Section \ref{Main result} our main result, i.e., the existence of a semiflow selection, cf. Theorem \ref{main result}.
		\item Section \ref{Weak Sequential Stability} will be entirely dedicated to prove the weak sequential stability of the solution set of \eqref{compressible viscous fluid} for a fixed initial data.
	\end{itemize}

	\section{Skorokhod space} \label{Skorokhod space}
	
	Given a separable Hilbert space $H$, we define the \textit{Skorokhod space} $\mathfrak{D}([0,\infty); H)$ as the space of the c\`{a}gl\`{a}d (an acronym for ``continue \`{a} gauche, limites \`{a} droite") functions defined on $[0,\infty)$ taking values in $H$. More precisely, $\Phi$ belongs to the space $\mathfrak{D}([0,\infty); H)$ if it is left--continuous and has right--hand limits:
	\begin{itemize}
		\item[(i)] for $t>0$, \ $\Phi(t-)=\lim_{s \uparrow t} \Phi(s)$ exists \ and \  $\Phi(t-)=\Phi(t)$;
		\item[(ii)] for $t\geq 0$, \  $\Phi(t+)=\lim_{s\downarrow t} \Phi(s)$ exists.
	\end{itemize}

	Unlike in the case of continuous functions, the topology on the space of c\`{a}gl\`{a}d functions on an unbounded interval cannot be built up by simply considering functions being c\`{a}gl\`{a}d on any compact, see e.g. Jakubowski \cite{Jak}. Instead, we proceed as follows.
	
	For every $\Phi\in  \mathfrak{D}([0,\infty); H)$ we define
	\begin{equation} \label{extension on [-1,0] and [T,T+1]}
		\widehat{\Phi}_T :=\begin{cases}
		\Phi(0) &\mbox{for } t\in [-1,0], \\
		\Phi(t) &\mbox{for } t\in (0,T), \\
		\Phi(T) &\mbox{for } t\in [T,T+1],
		\end{cases}
	\end{equation}
	and, fixing a basis $\{ e_k \}_{k\in \mathbb{N}}$ of the Hilbert space $H$, for every $\Phi, \Psi \in \widehat{\mathfrak{D}} ([0,\infty); H)$ we define
	\begin{equation} \label{metric non--compact, Hilbert}
		d_{\infty} (\Phi, \Psi) := \sum_{M=1}^{\infty} \sum_{k=1}^{\infty} \frac{1}{2^M}  \frac{1}{2^k} \frac{d_M\left( \left\langle \widehat{\Phi}_M; e_k\right\rangle, \left\langle \widehat{\Psi}_M; e_k \right\rangle \right)}{1+d_M\left( \left\langle \widehat{\Phi}_M; e_k\right\rangle, \left\langle \widehat{\Psi}_M; e_k \right\rangle \right)},
	\end{equation}
	where $d_M$ denotes the Skorokhod metric on the space $\mathfrak{D}([-1,M+1]; \mathbb{R})$, for all $M \in \mathbb{N}$; for the precise definition of $d_M$ and further details on the Skorokhod space we refer to Whitt \cite{Whi}, Chapter 12. It is worth noticing that, even if in \cite{Whi} the author considered c\`{a}gl\`{a}d (``continue \`{a} droite, limites \`{a} gauche") functions, the same construction works in our context as well: dealing with the completed graphs of the functions, which are obtained by adding segments joining the left and right limits at each discontinuity point to the graph, the actual value of the function at discontinuity points does not matter, provided that it falls appropriately between the left and right limits.
	
	It is easy to verify that $d_{\infty}$ is a metric on $\mathfrak{D}([0,\infty); H)$. Moreover, we have the following result.
	
		\begin{proposition} \label{characterization convergence non-compact hilbert}
		Let $\{ \Phi^n \}_{n\in \mathbb{N}}$ be a sequence in $\mathfrak{D}([0,\infty); H)$. If
		\begin{equation} \label{convergence non-compact} 
			\Phi^n \underset{n\rightarrow \infty}{\longrightarrow} \Phi \quad \mbox{in }\ \mathfrak{D}([0,\infty); H)
		\end{equation}
		then for all $k\in \mathbb{N}$
		\begin{equation} \label{almost everywhere convergence}
			\left\langle \Phi^n(t); e_k \right\rangle \underset{n\rightarrow \infty}{\longrightarrow} \left\langle \Phi(t); e_k \right\rangle \quad \mbox{for a.e. } t\in (0,\infty).
		\end{equation}
		Moreover
		\begin{itemize}
			\item[(i)] if $\Phi^n$ are monotone for all $n\in \mathbb{N}$, conditions \eqref{convergence non-compact} and \eqref{almost everywhere convergence} are equivalent;
			\item[(ii)] if $\Phi^n$ are continuous for all $n\in \mathbb{N}$, condition \eqref{convergence non-compact} is equivalent to
			\begin{equation*}
			\sup_{t\in [0,M]} \left| \left\langle \Phi^n(t)- \Phi(t) ; e_k \right \rangle  \right| \underset{n\rightarrow \infty}{\longrightarrow} 0
			\end{equation*}
			for all $k, M \in \mathbb{N}$.
		\end{itemize}
	\end{proposition}
	\begin{proof}
		First of all, we will show that \eqref{convergence non-compact} is equivalent to
		\begin{equation} \label{convergence compact}
			\left\langle \widehat{\Phi^n}_M; e_k \right\rangle \underset{n\rightarrow \infty}{\longrightarrow} \left\langle \widehat{\Phi}_M; e_k \right\rangle \quad \mbox{in }\ \mathfrak{D}([-1,M+1]; \mathbb{R})
		\end{equation}
		for all $k,M \in \mathbb{N}$. Indeed, if \eqref{convergence non-compact} holds, let $\varepsilon>0$, $k,M \in \mathbb{N}$ be fixed and choose a positive $\tilde{\varepsilon}=\tilde{\varepsilon}(\varepsilon)$ such that
		\begin{equation} \label{expression tilde epsilon}
			\frac{2^{k+M}\tilde{\varepsilon}}{1-2^{k+M}\tilde{\varepsilon}} < \varepsilon.
		\end{equation}
		From \eqref{convergence non-compact}, there exists $n_0=n_0(\tilde{\varepsilon})$ such that  
		\begin{equation*}
			\frac{1}{2^{k+M}} \frac{d_M\left( \left\langle \widehat{\Phi^n}_M; e_k\right\rangle, \left\langle \widehat{\Phi}_M; e_k \right\rangle \right)}{1+d_M\left( \left\langle \widehat{\Phi^n}_M; e_k\right\rangle, \left\langle \widehat{\Phi}_M; e_k \right\rangle \right)}\leq d_{\infty} (\Phi^n, \Phi) < \tilde{\varepsilon}, \quad \mbox{for all } n\geq n_0,
		\end{equation*}
		which, combined with \eqref{expression tilde epsilon}, implies
		\begin{equation*}
			d_M\left( \left\langle \widehat{\Phi^n}_M; e_k\right\rangle, \left\langle \widehat{\Phi}_M; e_k \right\rangle \right) < \varepsilon \quad \mbox{for all } n\geq n_0.
		\end{equation*}
		Vice versa, if \eqref{convergence compact} holds, let $\varepsilon>0$ be fixed and choose $N=N(\varepsilon)\geq 2$, such that $1/2^N <\varepsilon/2$. From \eqref{convergence compact}, there exists $n_0=n_0(\varepsilon)$ such that
		\begin{equation*}
			\max_{k+M \leq N} d_M\left( \left\langle \widehat{\Phi^n}_M; e_k\right\rangle, \left\langle \widehat{\Phi}_M; e_k \right\rangle \right) < \varepsilon \quad \mbox{for all } n\geq n_0.
		\end{equation*}
		For every $n\geq n_0$ we obtain
		\begin{equation*}
			d_{\infty} (\Phi^n, \Phi) \leq \sum_{2\leq k+M\leq N} \frac{\varepsilon}{2^{k+M}} + \sum_{k+M> N} \frac{1}{2^{k+M}} =  \frac{\varepsilon}{2} \left(1-\frac{1}{2^{N-1}}\right) + \frac{1}{2^N} <\varepsilon.
		\end{equation*}
		
		Let now \eqref{convergence compact} hold and let $k\in \mathbb{N}$ be fixed. From Theorem 12.5.1. in \cite{Whi}, for each $M \in \mathbb{N}$ 
		\begin{equation} \label{pointwise convergence in dense set}
			\begin{aligned}
				\left\langle \widehat{\Phi^n}_M(t); e_k \right\rangle &\underset{n\rightarrow \infty}{\longrightarrow} \left\langle \widehat{\Phi}_M(t); e_k \right\rangle \\
				\mbox{for all }t \mbox{ in a dense subset of }& [-1,M+1], \mbox{ including } -1 \mbox{ and } M+1,
			\end{aligned}
		\end{equation}
		implying in particular from \eqref{extension on [-1,0] and [T,T+1]}  the uniform convergence of $\widehat{\Phi^n}_M$ to $\widehat{\Phi}_M$ on $[-1,0]$ and $[M,M+1]$. We then recover that for all $M \in \mathbb{N}$
		\begin{equation} \label{convergence restricted functions}
			\left\langle \Phi^n|_{[0,M]}(t); e_k \right\rangle \underset{n\rightarrow \infty}{\longrightarrow} \left\langle \Phi|_{[0,M]}(t); e_k \right\rangle
		\end{equation}
		for all $t$ in a dense subset of $[0,M]$. More precisely, denoting with $Disc_M(\Phi)$ the set of discontinuities of $\Phi$ on $[0,M]$, specifically
		\begin{equation*}
			Disc_M(\Phi) := \{ t\in (0,M]: \ \Phi(t) \neq \Phi(t+) \}
		\end{equation*} 
		then, by Lemma 12.5.1 in \cite{Whi}, \eqref{convergence restricted functions} holds for all $t\notin Disc_M(\Phi)$. Introducing the set
		\begin{equation*}
			Disc(\Phi) = \bigcup_{M\in \mathbb{N}} Disc_M(\Phi),
		\end{equation*}
		then trivially $Disc_M(\Phi) \subseteq Disc_{M+1}(\Phi)$ for all $M\in \mathbb{N}$ and
		\begin{equation*}
			\left\langle \Phi^n(t); e_k \right\rangle \underset{n\rightarrow \infty}{\longrightarrow} \left\langle \Phi(t); e_k \right\rangle \quad \mbox{for all } t\notin Disc(\Phi);
		\end{equation*}
		since by Corollary 12.2.1 in \cite{Whi} each $Disc_M(\Phi)$ is either finite or countable, the set $Disc(\Phi)$ is at most countable and thus we get \eqref{almost everywhere convergence}.\\ 
		Furthermore, conditions (i) and (ii) follow easily from the fact that
		\begin{itemize}
			\item[(i)] if $\Phi^n$ is monotone for every $n\in \mathbb{N}$ then \eqref{convergence compact} is equivalent to \eqref{pointwise convergence in dense set} for all $k, M \in \mathbb{N}$ by \cite{Whi}, Corollary 12.5.1;
			\item[(ii)] if $\Phi^n$ is continuous for every $n\in \mathbb{N}$ then \eqref{convergence compact} reduces to uniform convergence on the interval $[-1, M+1]$ for all $k, M \in \mathbb{N}$.
		\end{itemize}
	\end{proof}
	
	\section{Semiflow selection} \label{Semiflow selection}
	In this section we will focus on proving the existence of a semiflow selection. We start by fixing our setting; from now on, $\mathcal{P}(X)$ will denote the family of all subsets of a space $X$. Let
	\begin{itemize}
		\item $H$ be a separable Hilbert space with a basis $\{e_k\}_{k\in \mathbb{N}}$;
		\item $D$ a closed convex subset of $H$;
		\item $\mathcal{T}= \mathfrak{D}([0,\infty); H)$;
		\item $\mathcal{U}: D \rightarrow \mathcal{P}(\mathcal{T})$ satisfy the following properties.
		\begin{itemize}
			\item[(\textbf{P1})] \textit{Non-emptiness}: for every $x \in D$, $\mathcal{U}(x)$ is a non-empty subset of $\mathcal{T}$.
			\item[(\textbf{P2})] \textit{Compactness}: for every $x\in D$, $\mathcal{U}(x)$ is a compact subset of $\mathcal{T}$.
			\item[(\textbf{P3})] \textit{Measurability}: the map $\mathcal{U}: D \rightarrow \mathcal{P}(\mathcal{T})$ is Borel measurable. 
			\item [(\textbf{P4})] \textit{Shift invariance}: introducing the positive shift operator $S_T \circ \Phi$ for every $T>0$ and $\Phi \in \mathcal{T}$ as
			\begin{equation*}
				S_T \circ \Phi(t) = \Phi(T+t), \quad \mbox{for all } t\geq 0,
			\end{equation*}
			then, for any $T>0$, $x\in D$ and $\Phi \in \mathcal{U}(x)$, we have
			\begin{equation*}
				S_T \circ \Phi \in \mathcal{U} (\Phi(T)).
			\end{equation*}
			\item[(\textbf{P5})] \textit{Continuation}: introducing the continuation operator $\Phi_1 \cup_T \Phi_2$ for every $T>0$ and $\Phi_1, \Phi_2 \in \mathcal{T}$ as
			\begin{equation*}
				\Phi_1 \cup_T \Phi_2(t) =\begin{cases}
					\Phi_1(t) &\mbox{for } 0\leq t\leq T, \\
					\Phi_2(t-T) &\mbox{for } t>T,
				\end{cases} \quad \mbox{for all }t\geq 0,
			\end{equation*}
			then, for any $T>0$, $x\in D$, $\Phi_1 \in \mathcal{U}(x)$ and $\Phi_2 \in \mathcal{U}(\Phi_1(T))$, we have
			\begin{equation*}
				\Phi_1 \cup_T \Phi_2 \in \mathcal{U}(x).
			\end{equation*}
		\end{itemize}
	\end{itemize}
	
	\begin{remark}
		It is worth noticing that, in order to guarantee the validity of the shift invariance and continuation properties, it is necessary to verify that for any $x \in D$ and any $\Phi \in \mathcal{U}(x)$, $\Phi$ evaluated at any $t\geq 0$ also belongs to set $D$.
	\end{remark}

	We are now ready to state and prove the following result.
	
	\begin{theorem} \label{existence semiflow selection}
		Let the mapping $\mathcal{U}: D\subseteq H \rightarrow \mathcal{P}(\mathcal{T})$ satisfy properties (\textbf{P1})--(\textbf{P5}) stated above. Then, there exists a Borel measurable map
		\begin{equation*}
			\mathfrak{u}: D \rightarrow \mathcal{T}, \quad \mathfrak{u}(x) \in \mathcal{U}(x) \mbox{ for every } x\in D,
		\end{equation*}
		satisfying the semigroup property: for any $x\in D$ and any $t_1, t_2 \geq 0$
		\begin{equation*}
			\mathfrak{u}(x)(t_1+t_2)= \mathfrak{u}[\mathfrak{u}(x)(t_1)](t_2).
		\end{equation*}
	\end{theorem}
	\begin{proof}
		The idea of the proof is to reduce iteratively the set $\mathcal{U}(x)$ for a fixed $x\in D$, selecting the minimum points of particular functionals in order to obtain finally a single point in $\mathcal{T}$, which will define $\mathfrak{u}(x)$. The procedure has been proposed by Cardona and Kapitanski \cite{CarKap} in the context of continuous trajectories and later adapted to more general setting in \cite{BreFeiHof} and \cite{Bas}.
		
		We introduce the functionals $I_{\lambda, k}: \mathcal{T}\rightarrow \mathbb{R}$ defined for every $\Phi \in \mathcal{T}$ as
		\begin{equation} \label{selection functional}
			I_{\lambda, k}(\Phi)=\int_{0}^{\infty} e^{-\lambda t} \  f\left(\left\langle \Phi(t); e_k \right\rangle\right) {\rm d}t,
		\end{equation}
		where $\lambda>0$, $\{e_k\}_{k\in \mathbb{N}}$ is a basis in $H$ and $f: \mathbb{R} \rightarrow \mathbb{R}$ is a fixed smooth, bounded and strictly increasing function; this choice is justified by the fact that for a fixed $k\in \mathbb{N}$ we can see $I_{\lambda,k}$ as the Laplace transform of the function $f\left(\left\langle \Phi(\cdot); e_k \right\rangle\right)$, an useful interpretation for the proof of the existence of the semiflow $\mathfrak{u}$.
		
		Let us first show the continuity of $I_{\lambda,k}$ for every $\lambda>0$ and $k\in \mathbb{N}$ fixed. Let $\{\Phi_n\}_{n\in \mathbb{N}}$ in $\mathcal{T}$ be such that
		\begin{equation} \label{convergence in Skorohod}
			\Phi_n \rightarrow \Phi \quad \mbox{in } \mathcal{T}
		\end{equation}
		as $n\rightarrow \infty$. By Proposition \ref{characterization convergence non-compact hilbert}, \eqref{convergence in Skorohod} implies in particular that
		\begin{equation*}
			\left \langle \Phi_n(t); e_k \right \rangle \rightarrow \left \langle \Phi(t); e_k \right \rangle \quad \mbox{for a.e. } t\in (0,\infty)
		\end{equation*}
		as $n\rightarrow \infty$. Due to continuity and boundedness of $f$, we get
		\begin{equation*}
			I_{\lambda,k}(\Phi_n) = \int_{0}^{\infty} e^{-\lambda t} \ f\left(\left\langle \Phi_n(t); e_k \right\rangle\right) {\rm d}t \rightarrow \int_{0}^{\infty} e^{-\lambda t} \ f\left(\left\langle \Phi(t); e_k \right\rangle\right) {\rm d}t = I_{\lambda, k}(\Phi),
		\end{equation*}
		i.e.,  what we wanted to prove.
		
		We define the selection mapping for every $x\in D$ as
		\begin{equation*}
			I_{\lambda, k} \circ \mathcal{U}(x) = \{ \Phi \in \mathcal{U}(x)  : \  I_{\lambda,k}(\Phi) \leq I_{\lambda, k}(\widetilde{\Phi}) \mbox{ for all } \widetilde{\Phi} \in \mathcal{U}(x)\}.
		\end{equation*}
		Notice, in particular, that the minimum exists since $I_{\lambda,k}$ is continuous on $\mathcal{T}$ and the set $\mathcal{U}(x)$ is compact. From the fact that $\mathcal{U}$ satisfies properties (\textbf{P1})--(\textbf{P5}), it is not difficult to show that the set--valued mapping
		\begin{equation*}
			I_{\lambda, k}\circ \mathcal{U}: D \rightarrow \mathcal{P}(\mathcal{T})
		\end{equation*}
		satisfies properties (\textbf{P1})--(\textbf{P5}) as well, cf. Proposition 5.1 in \cite{BreFeiHof}. 
		
		We are now ready to prove the existence of the semiflow selection $\mathfrak{u}$. Fixing a countable set $\{ \lambda_j \}_{j\in \mathbb{N}}$ dense in $(0,\infty)$, we can define the functionals 
		\begin{equation*}
			I_{j,k}(\Phi)=\int_{0}^{\infty} e^{-\lambda_j t} f\left(\left\langle \Phi(t); e_k \right\rangle\right) {\rm d}t.
		\end{equation*}
		Choosing an enumeration $\{ j(i), k(i) \}_{i=1}^{\infty}$ of the all involved combinations of indeces, we define the maps
		\begin{equation*}
			\mathcal{U}^i:= I_{j(i), k(i)} \circ \dots \circ I_{j(1), k(1)} \circ \mathcal{U}, \quad i=1,2, \dots,
		\end{equation*}
		and
		\begin{equation*}
			\mathcal{U}^{\infty} :=\bigcap_{i=1}^{\infty} \mathcal{U}^i.
		\end{equation*}
		It is easy to show that the set-valued map
		\begin{equation*}
			D \ni x \mapsto \mathcal{U}^{\infty}(x) \in \mathcal{P}(\mathcal{T})
		\end{equation*}
		satisfies properties (\textbf{P1})--(\textbf{P5}) as well; for details, see \cite{BreFeiHof}, Theorem 2.5 and \cite{Bas}, Theorem 2.2. 
		
		We now claim that for every $x\in D$ the set $\mathcal{U}^{\infty}(x)$ is a singleton. Indeed, if $\Phi_1, \Phi_2 \in \mathcal{U}^{\infty}(x)$ for a fixed $x\in D$, then
		\begin{equation*}
			I_{j(i), k(i)}(\Phi_1)=I_{j(i), k(i)}(\Phi_2)
		\end{equation*}
		for all $i=1,2,\dots$. Since the integrals $I_{j(i), k(i)}$ can be seen as Laplace transforms of the functions $f\left(\left\langle \Phi(\cdot); e_k \right\rangle\right)$, we can apply Lerch's theorem to deduce that
		\begin{equation*}
			f\left(\left\langle \Phi_1(t); e_k \right\rangle\right)=f\left(\left\langle \Phi_2(t); e_k \right\rangle\right)
		\end{equation*}
		for all $k\in \mathbb{N}$ and a.e. $t\in (0,\infty)$. Since the function $f$ is strictly increasing, we obtain that
		\begin{equation*}
			\left\langle \Phi_1(t); e_k \right\rangle = \left\langle \Phi_2(t); e_k \right\rangle
		\end{equation*}
		for all $k\in \mathbb{N}$ and a.e. $t\in (0,\infty)$; in particular, from \eqref{metric non--compact, Hilbert} we get that $d_{\infty}(\omega_1, \omega_2)=0$ and thus $\omega_1=\omega_2$ in $\mathcal{T}$.
		
		Finally, we define the semiflow selection $\mathfrak{u}$ for all $x\in D$ as
		\begin{equation*}
			\mathfrak{u}(x):= \mathcal{U}^{\infty}(x) \in \mathcal{T};
		\end{equation*}
		mesurability follows from the property $(\textbf{P3})$ for $\mathcal{U}^{\infty}$, while the semigroup property follows from property (\textbf{P4}): for any $x\in D$ and any $t_1, t_2 \geq 0$
		\begin{equation*}
		\mathfrak{u}(x)(t_1+t_2) = S_{t_1} \circ \mathfrak{u}(x)(t_2) =\mathfrak{u}(x(t_1))(t_2).
		\end{equation*}
	\end{proof}

	\section{Semiflow selection for compressible viscous fluids} \label{Compressible viscous fluids}
	
	Let us consider a general mathematical model of compressible viscous fluids, represented by the following system
	\begin{equation} \label{continuity equation}
		\partial_t \varrho + \divv_x (\varrho \textbf{u})=0,
	\end{equation}
	\begin{equation} \label{balance of momentum}
		\partial_t(\varrho \textbf{u}) + \divv_x(\varrho \textbf{u} \otimes \textbf{u})+ \nabla_x p = \divv_x \mathbb{S};
	\end{equation}
	here $\varrho=\varrho(t,x)$ denotes the density, $\textbf{u}=\textbf{u}(t,x)$ the velocity, $p=p(\varrho)$ the barotropic pressure and $\mathbb{S}$ the viscous stress tensor, which we suppose to be connected to the symmetric velocity gradient
	\begin{equation*}
		\mathbb{D}\textbf{u}= \frac{1}{2}(\nabla_x \textbf{u} + \nabla_x^T \textbf{u})
	\end{equation*}
	through the Fenchel's identity
	\begin{equation} \label{relation viscous stress symmetric gradient}
		\mathbb{S}: \mathbb{D}\textbf{u}=F(\mathbb{D}\textbf{u}) + F^*(\mathbb{S}),
	\end{equation}
	where, denoting with $\mathbb{R}^{d\times d}_{sym}$ the space of $d$-dimensional real symmetric tensors,
	\begin{equation} \label{conditions on F}
		F: \mathbb{R}^{d\times d}_{sym} \rightarrow 	[0,\infty) \mbox{ is a convex l.s.c. function with } F(0)=0,
	\end{equation}
	and $F^*$ is its conjugate, defined for every $\mathbb{A} \in \mathbb{R}^{d\times d}_{sym}$ as
	\begin{equation*}
		F^*(\mathbb{A}):= \sup_{\mathbb{B}\in 	\mathbb{R}^{d\times d}_{sym}} \{ \mathbb{A}: \mathbb{B}-F(\mathbb{B}) \}.
	\end{equation*}
	Notice that conditions \eqref{conditions on F} guarantees that
	\begin{equation} \label{conditions on F*}
		F^*: \mathbb{R}^{d\times d}_{\rm sym} \rightarrow [0,\infty] \mbox{ is a superlinear l.s.c. function}, \quad \liminf_{|\mathbb{S}|\rightarrow \infty} \frac{F^*(\mathbb{S})}{|\mathbb{S}|}=\infty.
	\end{equation}
	Furthermore, we will suppose $F$ to satisfy relation
	\begin{equation} \label{relation F and trsce-less part}
		F(\mathbb{D}) \geq \mu \left| \mathbb{D}-\frac{1}{d} \trace[\mathbb{D}]\mathbb{I} \right|^q \quad \mbox{for all } |\mathbb{D}|> 1,
	\end{equation}
	for some $\mu >0$ and $q>1$.
	Notice that condition \eqref{relation viscous stress symmetric gradient} is equivalent in requiring
	\begin{equation*}
		\mathbb{S} \in \partial F(\mathbb{D}\textbf{u}),
	\end{equation*}
	where $\partial$ denotes the subdifferential of a convex function. 
	
	Regarding pressure, we will consider the standard isentropic case
	\begin{equation} \label{pressure}
		p(\varrho)= a \varrho^{\gamma}, \quad \gamma\geq 1,
	\end{equation}
	with $a$ a positive constant; however, more general EOS preserving the essential features of \eqref{pressure} can be considered, cf. \cite{AbbFeiNov}. The pressure potential $P$, satisfying the ODE
	\begin{equation*}
		\varrho P'(\varrho)-P(\varrho)=p(\varrho),
	\end{equation*} 
	will be of the form
	\begin{equation} \label{pressure potential}
		P(\varrho) = \begin{cases}
			a \ \varrho \log \varrho &\mbox{if } \gamma=1, \\
			\frac{a}{\gamma-1} \varrho^{\gamma} &\mbox{if } \gamma >1;
		\end{cases}
	\end{equation}
	in particular, this implies that 
	\begin{equation} \label{conditions pressure potential}
		P \mbox{ is a strictly convex superlinear 	continuous function on } [0,\infty).
	\end{equation}
	
	We will study the system on the set $(0,\infty)\times \Omega$, where the physical domain $\Omega \subset \mathbb{R}^d$ is assumed to be bounded and Lipschitz, on the boundary of which we impose the no--slip condition
	\begin{equation} \label{boundary condition}
		\textbf{u}|_{\partial \Omega}=0.
	\end{equation} 
	Finally, we fix the initial conditions
	\begin{equation} \label{initial conditions}
		\varrho(0,\cdot)=\varrho_0, \quad (\varrho \textbf{u})(0,\cdot)= \textbf{m}_0.
	\end{equation}
	
	Our goal is to apply the abstract machinery introduced in the previous section in order to show the existence of a semiflow selection for system \eqref{continuity equation}--\eqref{initial conditions}. More precisely, we aim to prove Theorem \ref{main result} below, clarifying first the concept of solution we will work with and fixing a proper setting.
	
	\subsection{Dissipative solution} \label{Dissipative solution}
	As already mentioned in the introduction, inspired by the recent work of Abbatiello, Feireisl and Novotn\'{y} \cite{AbbFeiNov}, we will refer to the concept of dissipative solutions. From now  on, it is better to consider the density $\varrho$ and the momentum $\textbf{m}=\varrho \textbf{u}$ as state variables, since they are at least weakly continuous in time.
	
	\begin{definition} \label{dissipative solution}
		The pair of functions $[\varrho,\textbf{m}]$ constitutes a \textit{dissipative solution} to the problem \eqref{continuity equation}--\eqref{initial conditions} with the total energy $E$ and initial data
		\begin{equation*}
			[\varrho_0, \textbf{m}_0, E_0] \in L^{\gamma}(\Omega)\times L^{\frac{2\gamma}{\gamma+1}}(\Omega; \mathbb{R}^d) \times [0,\infty)
		\end{equation*}
		if the following holds:
		\begin{itemize}
			\item[(i)]  $\varrho \geq 0$ in $(0,\infty) \times \Omega$ and
			\begin{equation*}
				[\varrho, \textbf{m}, E] \in C_{\rm weak,loc}([0,\infty); L^{\gamma}(\Omega)) \times  C_{\rm weak,loc}([0,\infty); L^{\frac{2\gamma}{\gamma+1}}(\Omega;\mathbb{R}^d)) \times \mathfrak{D}([0,\infty));
			\end{equation*}
			\item[(ii)] the integral identity
			\begin{equation} \label{weak formulation continuity equation}
				\left[ \int_{\Omega} \varrho \varphi(t,\cdot) \ {\rm d}x \right]_{t=0}^{t=\tau}= \int_{0}^{\tau} \int_{\Omega} [\varrho \partial_t \varphi + \textbf{m}\cdot \nabla_x \varphi] \ {\rm d}x {\rm d}t
			\end{equation}
			holds for any $\tau >0$ and any $\varphi \in C_c^1([0,\infty)\times \overline{\Omega})$, with $\varrho(0,\cdot)=\varrho_0$;
			\item[(iii)] there exist
			\begin{equation*}
				\mathbb{S} \in L^1_{\rm loc} (0,\infty; L^1(\Omega; \mathbb{R}^{d\times d}_{\rm sym})) \quad \mbox{and} \quad \mathfrak{R} \in L^{\infty}_{\rm weak}(0,\infty; \mathcal{M}^+(\overline{\Omega}; \mathbb{R}^{d\times d }_{\rm sym}))
			\end{equation*}
			such that the integral identity
			\begin{equation} \label{weak formulation balance of momentum}
				\begin{aligned}
					\left[ \int_{\Omega} \textbf{m}\cdot \bm{\varphi}(t, \cdot) \ {\rm d}x \right]_{t=0}^{t=\tau} &= \int_{0}^{\tau}\int_{\Omega} \left[ \textbf{m} \cdot \partial_t \bm{\varphi} + \mathbbm{1}_{\varrho>0} \frac{\textbf{m}\otimes \textbf{m}}{\varrho}:\nabla_x \bm{\varphi} +p(\varrho)\divv_x\bm{\varphi}\right] \ {\rm d}x {\rm d}t \\
					&- \int_{0}^{\tau} \int_{\Omega} \mathbb{S}: \nabla_x \bm{\varphi} \ {\rm d}x {\rm d}t + \int_{0}^{\tau} \int_{\overline{\Omega}} \nabla_x \bm{\varphi} : {\rm d}\mathfrak{R} \ {\rm d}t
			\end{aligned}
			\end{equation}
			holds for any $\tau>0$ and any $\bm{\varphi} \in C^1_c([0,\infty)\times \overline{\Omega}; \mathbb{R}^d)$, $\bm{\varphi}|_{\partial \Omega}=0$, with $\textbf{m}(0,\cdot)=\textbf{m}_0$;
			\item[(iv)] there exists 
			\begin{equation*}
				\textbf{u} \in L^q_{\rm loc}(0,\infty; W_0^{1,q}(\Omega; \mathbb{R}^d)) \mbox{ such that } \textbf{m}=\varrho \textbf{u} \mbox{ a.e. in } (0,\infty)\times \Omega; 
			\end{equation*}
			\item[(v)] there exist a constant $\lambda>0$
			and a c\`{a}gl\`{a}d function $E$, non--increasing in $[0,\infty)$, satisfying
			\begin{equation} \label{energy}
				\int_{\Omega} \left[ \frac{1}{2} \frac{|\textbf{m}|^2}{\varrho} + P(\varrho) \right](\tau,\cdot) \ {\rm d}x + \frac{1}{\lambda}\int_{\overline{\Omega}} \rm d  \trace[\mathfrak{R}(\tau)] = E(\tau)
			\end{equation}
			for a.e. $\tau>0$, such that the energy inequality
			\begin{equation} \label{energy inequality}
				\left[ E(t)\psi(t) \right]_{t=\tau_1^-}^{t=\tau_2^+} -\int_{\tau_1}^{\tau_2} E \ \psi' \ {\rm d}t+ \int_{\tau_1}^{\tau_2} \psi \int_{\Omega} \left[ F(\mathbb{D}\textbf{u})+ F^*(\mathbb{S})\right] \ {\rm d}x {\rm d}t \leq 0
			\end{equation}
			holds for any $0\leq \tau_1 \leq \tau_2$ and any $\psi \in C_c^1[0,\infty)$, $\psi \geq 0$, with $E(0-)= E_0\geq E(0+)$.
		\end{itemize}
	\end{definition}
	
	\begin{remark} \label{remark on space of measures}
		In this context, $\mathcal{M}^+(\overline{\Omega})$ represents the space of all the positive Borel measures on $\overline{\Omega}$, while $\mathcal{M}^+(\overline{\Omega}; \mathbb{R}^{d\times d}_{\rm sym})$ denotes the space of tensor--valued (signed) Borel measures $\mathfrak{R}$ such that 
		\begin{equation*}
		\mathfrak{R}: (\xi \otimes \xi) \in \mathcal{M}^+(\overline{\Omega}),
		\end{equation*}
		for all $\xi \in \mathbb{R}^d$, and with components $\mathfrak{R}_{i,j}=\mathfrak{R}_{j,i}$. $L^{\infty}_{\rm weak}(0,\infty; \mathcal{M}(\overline{\Omega}))$ denotes the space of all the weak--$*$ measurable mapping $\nu: [0,\infty) \rightarrow \mathcal{M}(\overline{\Omega})$ such that
		\begin{equation*}
		\esssup_{t>0} \|\nu(t,\cdot) \|_{\mathcal{M}(\overline{\Omega})} <\infty,
		\end{equation*}
		which can also be identified as the dual space of $L^1(0,\infty; C(\overline{\Omega}))$.
	\end{remark}

	\subsubsection{Short remark on the Reynolds stress} \label{remark reynold stress}
	
	The concentration measure $\mathfrak{R}$, that we may call \textit{Reynolds stress}, appearing in the weak formulation of the balance of momentum \eqref{weak formulation balance of momentum} arises from possible oscillations and/or concentrations in the convective and pressure terms
	\begin{equation*}
		\mathbbm{1}_{\varrho>0} \frac{\textbf{m}\otimes \textbf{m}}{\varrho} + p(\varrho) \mathbb{I}
	\end{equation*}
	when $\gamma>1$, while for $\gamma=1$, i.e. when the pressure is a linear function of the density $\varrho$, it is only the convective term that contributes to $\mathfrak{R}$. 
	
	By consistency, as clearly explained in \cite{AbbFeiNov}, we should have introduced the dissipation defect $\mathfrak{E} \in L^{\infty}_{\rm weak}(0,\infty; \mathcal{M}^+(\overline{\Omega}))$ of the total energy arising from possible concentrations and/or oscillations in the kinetic and potential energy terms
	\begin{equation*}
		\frac{1}{2} \frac{|\textbf{m}|^2}{\varrho} +P(\varrho).
	\end{equation*}
	Instead of \eqref{energy} we would then have
	\begin{equation} \label{energy with dissipation defect}
		E(\tau) = \int_{\Omega} \left[ \frac{1}{2} \frac{|\textbf{m}|^2}{\varrho} + P(\varrho) \right](\tau,\cdot) \ {\rm d}x + \int_{\overline{\Omega}} \rm d  \mathfrak{E}(\tau),
	\end{equation}
	satisfying the energy inequality \eqref{energy inequality}. Choosing a positive constant $\lambda>0$ such that
	\begin{equation*}
		\trace \left[ \mathbbm{1}_{\varrho>0}\frac{\textbf{m}\otimes \textbf{m}}{\varrho}+ \chi(\gamma)p(\varrho)  \mathbb{I} \right] =\frac{|\textbf{m}|^2}{\varrho} +d(\gamma-1) P(\varrho) \leq \lambda \left(\frac{1}{2} \frac{|\textbf{m}|^2}{\varrho} +P(\varrho)\right)
	\end{equation*}
	with
	\begin{equation*}
		\chi(\gamma)=\begin{cases}
			0 & \mbox{if } \gamma=1, \\
			1 &\mbox{if } \gamma>1,
		\end{cases}
	\end{equation*}
	and adapting Lemma 1.6 in \cite{Bas1}, we could recover the compatibility condition
	\begin{equation} \label{dissipation defect total energy}
	\frac{1}{\lambda} \trace [\mathfrak{R}(\tau)] \leq \mathfrak{E}(\tau) \quad \mbox{for a.e. } \tau>0.
	\end{equation}
	
	In this sense, our choice of the energy \eqref{energy} makes the problem more general and easier to handle with only one free quantity instead of two; however, it reduces to \eqref{energy with dissipation defect} simply choosing a dissipation defect $\mathfrak{E}$ of the type
	\begin{equation*}
		\mathfrak{E}(\tau) := \frac{1}{\lambda} \trace [\mathfrak{R}(\tau)] \quad \mbox{for a.e. } \tau>0.
	\end{equation*}
	
	\subsection{Set--up} \label{Set--up}
	First of all, we must fix the space $H$, the subset $D\subseteq H$ and the map $\mathcal{U}$ introduced at the beginning of Section \ref{Semiflow selection}. In this context
	\begin{itemize}
		\item $H:= W^{-k, 2}(\Omega) \times W^{-k, 2}(\Omega; \mathbb{R}^d) \times \mathbb{R}$, where the natural number $k> \frac{d}{2} +1$ is fixed;
		\item $D$ represents the space of initial data; it can be chosen as
		\begin{equation*}
			D := \left\{ [\varrho_0, \textbf{m}_0, E_0] \in H: \ \varrho_0 \in L^1(\Omega), \ \varrho_0\geq 0, \ \textbf{m}_0\in L^1(\Omega;\mathbb{R}^d) \mbox{ satisfying } \eqref{initial energy}
		\right\}
		\end{equation*}
		where 
		\begin{equation} \label{initial energy}
			\int_{\Omega} \left[ \frac{1}{2} \frac{|\textbf{m}_0|^2}{\varrho_0} + P(\varrho_0)\right] {\rm d}x \leq E_0;
		\end{equation}
		\item $\mathcal{T}=\mathfrak{D}([0,\infty); H)$ represents the trajectory space;
		
		\item $\mathcal{U}: D\rightarrow \mathcal{P}(\mathcal{T})$ represents the set--valued mapping that associate to every $[\varrho_0, \textbf{m}_0, E_0] \in D$ the family of dissipative solutions in the sense of Definition \ref{dissipative solution} arising from the initial data $[\varrho_0, \textbf{m}_0, E_0] $. More precisely, for every $[\varrho_0, \textbf{m}_0, E_0] \in D$
		\begin{align*}
			\mathcal{U}&[\varrho_0, \textbf{m}_0, E_0]= \\
			&\{ [\varrho, \textbf{m}, E] \in \mathcal{T}: \ [\varrho, \textbf{m}, E] \mbox{ is a dissipative solution with initial data } [\varrho_0, \textbf{m}_0, E_0] \}.
		\end{align*}
	\end{itemize}

	Notice that everything is well-defined; indeed, denoting with $L^1_+(\Omega)$ the space of non--negative integrable functions on $\Omega$, we can rewrite $D$ as
	\begin{equation*}
		\{ [\varrho_0, \textbf{m}_0, E_0] \in L^1_+(\Omega)\times L^1(\Omega;\mathbb{R}^d)\times \mathbb{R}: \ g(\varrho_0, \textbf{m}_0) \leq E_0 \},
	\end{equation*}
	so that it coincides with the epigraph of the function $g: L^1_+(\Omega)\times L^1(\Omega;\mathbb{R}^d) \rightarrow [0,+\infty]$ defined as
	\begin{equation*}
		g(\varrho_0, \textbf{m}_0) = \int_{\Omega} \left[ \frac{1}{2} \frac{|\textbf{m}_0|^2}{\varrho_0} + P(\varrho_0)\right] dx.
	\end{equation*}
	From \eqref{conditions pressure potential} and the fact that
	\begin{equation*}
		[\varrho, \textbf{m}] \mapsto \begin{cases}
			0 &\mbox{if }\textbf{m}=0, \\
			\frac{|\textbf{m}|^2}{\varrho} &\mbox{if } \varrho >0, \\
			\infty &\mbox{otherwise},
		\end{cases}
	\end{equation*}
	we get that the function $g$ is lower semi--continuous and convex and thus its epigraph is a closed convex subset of $L^{\gamma}(\Omega)\times L^{\frac{2\gamma}{\gamma+1}}(\Omega;\mathbb{R}^d)\times \mathbb{R}$ for all $\gamma\geq 1$.
	
	From our choice of $k$, we can use the Sobolev embedding
	\begin{equation} \label{Sobolev embedding}
		L^r(\Omega) \hookrightarrow\hookrightarrow W^{-k,2}(\Omega) \quad \mbox{for every } r\geq 1
	\end{equation}
	to conclude that
	\begin{equation*}
		C_{\rm weak, loc} ([0,\infty); L^r(\Omega)) \hookrightarrow C_{\rm loc} ([0,\infty); W^{-k,2}(\Omega)) \hookrightarrow \mathfrak{D}([0,\infty); W^{-k,2}(\Omega)),
	\end{equation*}
	for every $r\geq 1$. Furthermore, due to the weak continuity of the density $\varrho$ and the momentum $\textbf{m}$, for every fixed $T>0$ and every $t\in [0,T]$, from the energy inequality we can deduce that
	\begin{equation*}
		\|\varrho(t, \cdot)\|_{L^{\gamma}(\Omega)} \leq \sup_{t\in [0,T]} \| \varrho(t, \cdot) \|_{L^{\gamma}(\Omega)} \leq c\sup_{t\in [0,T]} \| 1+P(\varrho) (t,\cdot)\|_{L^1(\Omega)} \leq c(E_0, \Omega),
	\end{equation*}
	\begin{align*}
		\|\textbf{m}(t, \cdot)\|_{L^{\frac{2\gamma}{\gamma+1}}(\Omega; \mathbb{R}^d)} &\leq \sup_{t\in [0,T]} \| \textbf{m}(t, \cdot) \|_{L^{\frac{2\gamma}{\gamma+1}}(\Omega; \mathbb{R}^d)} \\
		&\leq \esssup_{t\in (0,T)} \left \| \frac{\textbf{m}}{\sqrt{\varrho}}(t,\cdot) \right\|_{L^2(\Omega; \mathbb{R}^d)} \|\sqrt{\varrho}(t, \cdot)\|_{L^{2\gamma}(\Omega)}\leq c(E_0, \Omega);
	\end{align*}
	
	Finally, from condition (i) of Definition \ref{dissipative solution} we also have that $\varrho(t,\cdot)\geq 0$ for all $t\geq 0$, while relation
	\begin{equation*}
		\int_{\Omega} \left[ \frac{1}{2} \frac{|\textbf{m}|^2}{\varrho}+ P(\varrho) \right](t,\cdot) dx \leq E(t-)=E(t)
	\end{equation*}
	holds for all $t\geq 0$ since the energy is convex and $\varrho$ and $\textbf{m}$ are weakly continuous in time. In particular, we have that for every $t\geq 0$
	\begin{equation*}
		[\varrho(t,\cdot), \textbf{m}(t, \cdot), E(t)] \in D.
	\end{equation*}
	
	\subsection{Main result} \label{Main result}
	Keeping in mind the notation introduced in the previous section, we are now ready to state our main result.
	
	\begin{theorem} \label{main result}
		System \eqref{continuity equation}--\eqref{initial conditions} admits a semiflow selection $U$ in the class of dissipative solutions, i.e., there exists a Borel measurable map $U: D \rightarrow \mathcal{T}$ such that
		\begin{equation*}
			U[\varrho_0, \textbf{m}_0, E_0] \in \mathcal{U}[\varrho_0, \textbf{m}_0, E_0] \mbox{ for every } [\varrho_0, \textbf{m}_0, E_0] \in D
		\end{equation*}
		satisfying the semigroup property: for any $[\varrho_0, \textbf{m}_0, E_0] \in D$ and any $t_1, t_2 \geq 0$
		\begin{equation*}
			U[\varrho_0, \textbf{m}_0, E_0] (t_1+t_2)= U[\varrho(t_1), \textbf{m}(t_1), E(t_1)] (t_2)
		\end{equation*}
		where $[\varrho, \textbf{m}, E]= U[\varrho_0, \textbf{m}_0, E_0]$.
	\end{theorem}
	
	Theorem \ref{main result} is a consequence of Theorem \ref{existence semiflow selection} once we have verified that $\mathcal{U}$ satisfies properties (\textbf{P1})--(\textbf{P5}). To this end, we emphasise the following points.
	\begin{itemize}
		\item Property (\textbf{P1}) is equivalent in showing the \textit{existence} of a dissipative solution in the sense of Definition \ref{dissipative solution} for any fixed initial data $[\varrho_0, \textbf{m}_0, E_0] \in D$. For $\gamma>1$, this is the main result achieved in \cite{AbbFeiNov}, Section 3, while the case $\gamma=1$ can be done as well applying Lemma \ref{auxiliary lemma} below instead of Lemma 8.1 in \cite{AbbFeiNov}.
		\item Properties (\textbf{P2}) and (\textbf{P3}) hold true if we manage to prove the \textit{weak sequential stability} of the solution set $\mathcal{U}[\varrho_0, \textbf{m}_0, E_0]$ for every $[\varrho_0, \textbf{m}_0, E_0] \in D$ fixed, since it will  in particular imply compactness and the closed-graph property of the mapping
		\begin{equation*}
			D \ni [\varrho_0, \textbf{m}_0, E_0]  \rightarrow \mathcal{U}[\varrho_0, \textbf{m}_0, E_0] \in \mathcal{P}(\mathcal{T}),
		\end{equation*}
		and thus the Borel--measurality of $\mathcal{U}$, cf. Lemma 12.1.8 in \cite{StrVar}.
		\item Properties (\textbf{P4}) and (\textbf{P5}) can be easily checked following the same arguments done in \cite{BreFeiHof}, Lemma 4.2 and 4.3.
	\end{itemize}
	In conclusion, we are done if we show the weak sequential stability of the solution set $\mathcal{U}[\varrho_0, \textbf{m}_0, E_0]$ for every $[\varrho_0, \textbf{m}_0, E_0] \in D$ fixed. Being the proof quite elaborated, it is postponed to the next section.
	
	\begin{remark}
		As already done for the Euler and Navier--Stokes systems, cf. \cite{BreFeiHof}, \cite{Bas}, among all the dissipative solutions emanating from the same initial data it is possible to select only the \textit{admissible} ones, i.e., satisfying the physical principal of minimizing the total energy or equivalently, that are minimal with respect to relation $\prec$ defined as
		\begin{equation*}
			[\varrho^1, \textbf{m}^1, E^1] \prec [\varrho^2, \textbf{m}^2, E^2] \quad \Leftrightarrow \quad E^1(\tau\pm) \leq E^2(\tau\pm) \ \mbox{ for any }\tau \in (0,\infty),
		\end{equation*}
		where $[\varrho^i, \textbf{m}^i, E^i]$, $i=1,2$ are two dissipative solutions sharing the same initial data. Indeed, it is sufficient to start the selection considering in \eqref{selection functional} the functional $I_{1,k}$ with the function $f$ such that
		\begin{equation*}
			f \left( \left\langle \Phi(t) ; e_k \right\rangle \right) = f(E(t))  \quad \mbox{for all } t\geq 0,
		\end{equation*}
		where $\Phi(t)= [\varrho(t), \textbf{m}(t), E(t)]$; see \cite{BreFeiHof}, Lemma 5.2 for more details.
	\end{remark}

	\section{Weak sequential stability} \label{Weak Sequential Stability}
	
	This section will be entirely dedicated to the proof of the following result.
	
	\begin{proposition}
		Let $\{ [\varrho_n, \textbf{m}_n] \}_{n \in \mathbb{N}}$ be a  family of dissipative solutions with the corresponding total energies  $\{ E_n\}_{n \in \mathbb{N}}$ and initial data $\{ [\varrho_{0,n}, \textbf{m}_{0,n}, E_{0,n}] \}_{n \in \mathbb{N}}$ in the sense of Definition \ref{dissipative solution}. If 
		\begin{equation*}
			[\varrho_{0,n}, \textbf{m}_{0,n}, E_{0,n}] \rightarrow [\varrho_0, \textbf{m}_0, E_0] \quad \mbox{in } L^{\gamma}(\Omega) \times L^{\frac{2\gamma}{\gamma+1}}(\Omega; \mathbb{R}^d) \times \mathbb{R},
		\end{equation*}
		then, at least for suitable subsequences,
		\begin{equation} \label{convergence in trajectory space}
			[\varrho_n, \textbf{m}_n. E_n] \rightarrow [\varrho, \textbf{m}, E] \quad \mbox{in } \ \mathfrak{D}([0, \infty); W^{-k, 2}(\Omega) \times W^{-k, 2}(\Omega; \mathbb{R}^d) \times \mathbb{R}),
		\end{equation}
		where the natural number $k>\frac{d}{2}+1$ is fixed and $[\varrho, \textbf{m}]$ is another dissipative solution of the same problem with total energy $E$.
	\end{proposition}

	The proof will be divided in four steps:
	\begin{enumerate}
		\item in Section \ref{Convergences} we will first deduce a family of uniform bounds and convergences, including the limits $\varrho$ of the densities, $\textbf{m}$ of the momenta and $\textbf{u}$ of the velocities;
		\item in Section \ref{Limit passage} we will pass to the limit in the weak formulation of the continuity equation and the balance of momentum;
		\item in order to show that $\textbf{m}$ can be written as the product $\varrho \textbf{u}$, in Section \ref{Auxuliary lemma} we will state and prove Lemma \ref{auxiliary lemma};
		\item finally, in Section \ref{Energy convergence} we will focus on finding the limit $E$ of the energies.
	\end{enumerate}

	\subsection{Uniform bounds and  limits establishment} \label{Convergences}
	
	Our first goal is to show the following convergences, passing to suitable subsequences as the case may be:
	\begin{align}
		\varrho_n \rightarrow \varrho \quad &\mbox{in } C_{\rm weak, loc}([0,\infty); L^{\gamma}(\Omega)), \label{convergence densities} \\
		\textbf{m}_n \rightarrow \textbf{m} \quad &\mbox{in } C_{\rm weak, loc}([0,\infty); L^{\frac{2\gamma}{\gamma+1}}(\Omega; \mathbb{R}^d)), \label{convergence momenta} \\
		\textbf{u}_n \rightharpoonup \textbf{u} \quad &\mbox{in } L^q_{\rm loc}(0,\infty; W^{1,q}_0(\Omega; \mathbb{R}^d)) \label{convergence velocities}\\
		\mathbb{S}_n \rightharpoonup \mathbb{S} \quad &\mbox{in } L^1_{\rm loc}(0,\infty; L^1(\Omega; \mathbb{R}^{d\times d})), \label{convergence viscous stress tensor}\\
		\mathbbm{1}_{\varrho_n>0}\frac{\textbf{m}_n \otimes \textbf{m}_n}{\varrho_n} \overset{*}{\rightharpoonup} \overline{\mathbbm{1}_{\varrho>0} \frac{\textbf{m}\otimes \textbf{m}}{\varrho}} \quad &\mbox{in }L^{\infty}_{\rm weak}(0,\infty; \mathcal{M}(\overline{\Omega}; \mathbb{R}^{d\times d}_{\rm sym})), \label{convergence convective terms}\\
		p(\varrho_n) \overset{*}{\rightharpoonup}
		\overline{p(\varrho)} \quad &\mbox{in } 
		L^{\infty}_{\rm weak}(0,\infty; 	\mathcal{M}(\overline{\Omega})) \quad\mbox{if }\gamma>1, \label{convergence pressure terms}\\
		\mathfrak{R}_n \overset{*}{\rightharpoonup} \widetilde{\mathfrak{R}}  \quad &\mbox{in } L^{\infty}_{\rm weak}(0,\infty; \mathcal{M}^+(\overline{\Omega}; \mathbb{R}^{d\times d}_{\rm sym})) \label{convergence defects}.
	\end{align}
	From our hypothesis, all the initial energies are uniformly bounded by a positive constant $\overline{E}$ independent of $n$; specifically,
	\begin{equation*}
		\int_{\Omega} \left[ \frac{1}{2} \frac{|\textbf{m}_{0,n}|^2}{\varrho_{0,n}} + P(\varrho_{0,n}) \right] dx \leq \overline{E}.
	\end{equation*}
	From \eqref{energy} and the energy inequality \eqref{energy inequality} it is easy to deduce the following uniform bounds
	\begin{align}
		\esssup_{t>0} \left\| \frac{\textbf{m}_n}{\sqrt{\varrho_n}} (t,\cdot) \right \| _{L^2(\Omega; \mathbb{R}^d)} &\leq c_1=c_1(\overline{E}), \label{estimate kinetic energy} \\
		\esssup_{t>0} \| P(\varrho_n)(t,\cdot)\|_{L^1(\Omega)} & \leq c(\overline{E}), \label{estimate pressure potential}\\
		\esssup_{t>0} \|\trace[\mathfrak{R}_n(t)]\|_{\mathcal{M}^+(\overline{\Omega})} &\leq c(\overline{E}), \label{estimate defect energy}\\
		\| F (\mathbb{D} \textbf{u}_n) \|_{L^1((0,\infty)\times \Omega)} &\leq c(\overline{E}), \label{estimate F} \\
		\| F^* (\mathbb{S}_n) \|_{L^1((0,\infty)\times \Omega)} &\leq c(\overline{E}). \label{estimate F*}
	\end{align}
	
	\subsubsection{Convergences of $\varrho_n$ and $\textbf{m}_n$}
	For $\gamma>1$, from \eqref{pressure potential} and \eqref{estimate pressure potential} we can easily deduce, passing to a suitable subsequence as the case may be,
	\begin{equation} \label{convergence densities gamma>1}
		\varrho_n \overset{*}{\rightharpoonup} \varrho \quad \mbox{in } L^{\infty}(0,\infty; L^{\gamma}(\Omega)).
	\end{equation}
	Similarly, from \eqref{convergence densities gamma>1}, \eqref{estimate kinetic energy} and the fact that for a.e. $t>0$
	\begin{equation*}
		\|\textbf{m}_n(t,\cdot)\|_{L^{\frac{2\gamma}{\gamma+1}}(\Omega;\mathbb{R}^d)} \leq \left\| \frac{\textbf{m}_n}{\sqrt{\varrho_n}} (t,\cdot) \right \| _{L^2(\Omega; \mathbb{R}^d)}\| \sqrt{\varrho_n}(t,\cdot)\|_{L^{2\gamma}(\Omega)} \leq c(\overline{E}),
	\end{equation*}
	passing to a suitable subsequence, we obtain
	\begin{equation} \label{convergence momenta gamma>1}
		\textbf{m}_n \overset{*}{\rightharpoonup} \textbf{m} \quad \mbox{in } L^{\infty}(0,\infty; L^{\frac{2\gamma}{\gamma+1}}(\Omega; \mathbb{R}^d)).
	\end{equation}
	
	Since the $L^1$-space is not reflexive, for $\gamma=1$ a more detailed analysis is needed. If we consider the Young function $\Phi(z)=z\log^+z$, the densities $\{ \varrho_n \}_{n\in \mathbb{N}}$ can be seen as uniformly bounded in $L^{\infty}(0,\infty;L_{\Phi}(\Omega))$, where $L_{\Phi}(\Omega)$ is the Orlicz space associated to $\Phi$; indeed, noticing that
	\begin{equation*}
		\varrho\log^+ \varrho= \begin{cases}
			0 &\mbox{if } 0\leq \varrho <1, \\
			\varrho \log \varrho &\mbox{if } \varrho \geq 1,
		\end{cases}
		\quad\mbox{and} \quad  -\frac{1}{e} \leq \varrho\log \varrho\leq 0 \quad \mbox{if } 0\leq \varrho \leq 1,
	\end{equation*}
	from \eqref{estimate pressure potential}, for a.e. $\tau >0$ we have
	\begin{align*}
		\int_{\Omega}\varrho \log^+ \varrho (\tau,\cdot) \ dx &= \int_{\{ \varrho \geq 1 \}} \varrho \log \varrho(\tau, \cdot) \ dx \\
		&\leq \int_{\Omega} \varrho \log \varrho (\tau, \cdot) \ dx - \int_{\{ 0\leq \varrho < 1 \}} \varrho \log \varrho (\tau,\cdot) \ dx \\
		&\leq c(\overline{E}) + \frac{1}{e} |\Omega|.
	\end{align*}
	As the function $\Phi$ satisfies the $\Delta_2$--condition, $L_{\Phi}(\Omega)$ can be seen as the dual space of the Orlicz space $L_{\Psi}(\Omega)$, where $\Psi$ denotes the complementary Young function of $\Phi$, and hence,  passing to a suitable subsequence, we get
	\begin{equation} \label{convergence densities gamma=1}
	\varrho_n \overset{*}{\rightharpoonup} \varrho \quad \mbox{in } L^{\infty}(0,\infty;L_{\Phi}(\Omega)).
	\end{equation}
	
	We are  now able to prove the uniform integrability of the sequence $\{\textbf{m}_n(t,\cdot)\}_{n\in \mathbb{N}} \subset L^1(\Omega; \mathbb{R}^d)$ for a.e. $t>0$. More precisely, we want to show that for every $\varepsilon>0$ there exists $\delta=\delta(\varepsilon)$ such that for all $n\in \mathbb{N}$ and a.e. $t>0$
	\begin{equation*}
		\int_{M} |\textbf{m}_n|(t,\cdot) dx < \varepsilon \quad \mbox{for all } M\subset \Omega \mbox{ such that } |M|<\delta.
	\end{equation*}
	Fix $\varepsilon>0$ and choose $\tilde{\varepsilon}= \tilde{\varepsilon}(\varepsilon)$ such that $\tilde{\varepsilon} = (\varepsilon / c_1)^2$, with $c_1$ as in \eqref{estimate kinetic energy}. The superlinearity of $P$ \eqref{conditions pressure potential} combined with the uniform bound \eqref{estimate pressure potential} guarantees that the sequence $\{ \varrho_n(t,\cdot) \}_{n\in \mathbb{N}} \subset L^1(\Omega)$ is uniformly integrable for a.e. $t>0$, as a consequence of de la Vall\'{e}e--Poussin criterion; thus, there exists $\delta= \delta(\tilde{\varepsilon})$ such that for all $n\in \mathbb{N}$ and a.e. $t>0$, 
	\begin{equation} \label{equintegrability densities}
		\int_{M} \varrho_n(t,\cdot) dx < \tilde{\varepsilon} \quad \mbox{for all } M\subset \Omega \mbox{ such that } |M|<\delta.
	\end{equation}
	Fix $M \subset \Omega$ with $|M|< \delta$; applying H\"{o}lder's inequality, \eqref{estimate kinetic energy}, \eqref{equintegrability densities} and writing
	\begin{equation*}
		\textbf{m} = \sqrt{\varrho} \ \frac{\textbf{m}}{\sqrt{\varrho}},
	\end{equation*}
	we get that for all $n\in \mathbb{N}$ and a.e. $t>0$ 
	\begin{equation*}
	\int_{M} |\textbf{m}_n|(t,\cdot) dx \leq \left( \int_{M} \varrho_n(t,\cdot) dx \right)^{\frac{1}{2}} \left( \int_{M} \frac{|\textbf{m}_n|^2}{\varrho_n}(t,\cdot) dx \right)^{\frac{1}{2}} < c_1 \tilde{\varepsilon}^{\frac{1}{2}} =\varepsilon.
	\end{equation*}
	Dunford--Pettis theorem ensures that for a.e. $t>0$ the sequence $\{\textbf{m}_n(t,\cdot)\}_{n\in \mathbb{N}} \subset L^1(\Omega; \mathbb{R}^d)$ is relatively compact with respect to the weak topology; in particular, we have that
	\begin{equation*}
		\textbf{m}_n \rightharpoonup \textbf{m} \quad \mbox{in } L^1_{\rm loc} (0,\infty; L^1(\Omega; \mathbb{R}^d)).
	\end{equation*}
	
	Next, to get \eqref{convergence densities} from \eqref{convergence densities gamma>1} and \eqref{convergence densities gamma=1} we have to show that the family of $t$--dependent functions
	\begin{equation*}
		f_n(t) := \int_{\Omega} \varrho_n(t,\cdot) \phi dx
	\end{equation*}
	converges strongly in $C([a,b])$ for any $\phi \in C_c^{\infty}(\Omega)$ and any compact subset $[a,b]\subset (0,\infty)$. Recalling that the densities $\varrho_n$ and the momenta $\textbf{m}_n$ are weakly continuous in time, the sequences $\{ f_n \}_{n\in \mathbb{N}}$ and $\{ f_n' \}_{n\in \mathbb{N}}$ are uniformly bounded in $[a,b]$, since for all $\gamma \geq 1$ 
	\begin{equation*}
		\sup_{t\in [a,b]} |f_n(t)| \leq \sup_{t\in [a,b]}\|\varrho_n(t,\cdot)\|_{L^{\gamma}(\Omega)} \|\phi\|_{L^{\gamma'}(\Omega)} \leq c(\phi),
	\end{equation*}
	while from the uniform boundedness of the momenta $\textbf{m}_n$ in $L^{\infty}(0,\infty; L^p(\Omega; \mathbb{R}^d))$ with $p=\frac{2\gamma}{\gamma+1}$ and $\gamma\geq 1$,
	\begin{equation*}
	\sup_{t\in [a,b]} |f_n'(t)| \leq \sup_{t\in [a,b]} \| 	\textbf{m}_n (t,\cdot)\|_{L^p(\Omega:\mathbb{R}^d)} \|\nabla_x \phi\|_{L^{p'}(\Omega;\mathbb{R}^d)} \leq c(\phi).
	\end{equation*}
	As a consequence of the Arzel\`{a}-Ascoli theorem, we get \eqref{convergence densities}. A similar argument can be applied to get \eqref{convergence momenta}.
	
	Finally, notice that, from the Sobolev embedding \eqref{Sobolev embedding}, $\varrho_n: [0,\infty) \rightarrow W^{-k,2}(\Omega)$ is a continuous function for all $n\in \mathbb{N}$, and thus by Proposition \ref{characterization convergence non-compact hilbert} showing
	\begin{equation*}
		\varrho_n \rightarrow \varrho \quad \mbox{in } \mathfrak{D}([0,\infty); W^{-k,2}(\Omega))
	\end{equation*}
	is equivalent to 
	\begin{equation*}
		\sup_{t\in [0,M]} \left| \int_{\Omega} (\varrho_n-\varrho)(t,\cdot) e_k dx \right|  \rightarrow 0, \quad \mbox{for all } k,M \in \mathbb{N},
	\end{equation*}
	where $\{ e_k \}_{k\in \mathbb{N}}$ is an orthonormal basis of $L^2(\Omega)$; but this easily follows from convergence \eqref{convergence densities} and Parseval's identity. The same argument can be applied to show that
	\begin{equation*}
		\textbf{m}_n \rightarrow \textbf{m} \quad \mbox{in } \mathfrak{D}([0,\infty); W^{-k,2}(\Omega; \mathbb{R}^d)).
	\end{equation*}
	
	\subsubsection{Convergences of $\textbf{u}_n$ and $\mathbb{S}_n$}
	
	From \eqref{relation F and trsce-less part} and \eqref{estimate F} we can also deduce that
	\begin{equation*}
	\left\| \mathbb{D}\textbf{u}_n -\frac{1}{d} (\divv_x \textbf{u}_n) \mathbb{I}\right\|_{L^q((0,\infty)\times \Omega;\mathbb{R}^{d\times d})} \leq c(\overline{E}).
	\end{equation*}
	Fixing a compact interval $[a,b] \subset (0,+\infty)$ and an open bounded interval $I$ such that $[a,b] \subset I$, the previous inequality combined with the $L^q$-version of the trace--free Korn's inequality (see for instance \cite{BreCiaDie}, Theorem 3.1) implies
	\begin{equation*}
	\| \nabla_x \textbf{u}_n\|_{L^q(I\times \Omega; \mathbb{R}^{d\times d})} \leq c(\overline{E});
	\end{equation*}
	the standard Poincar\'{e} inequality ensures then
	\begin{equation*}
	\| \textbf{u}_n \|_{L^q(I; W_0^{1,q}(\Omega;\mathbb{R}^d))} \leq c(\overline{E}),
	\end{equation*}
	and thus we get convergence \eqref{convergence velocities}.
	
	The superlinearity of $F^*$ \eqref{conditions on F*} combined with \eqref{estimate F*}, the de la Vall\'{e}e--Poussin criterion and the Dunford--Pettis theorem, gives convergence \eqref{convergence viscous stress tensor}.
	
	\subsubsection{Convergences of $p(\varrho_n)$, $\frac{\textbf{m}_n \otimes \textbf{m}_n}{\varrho_n}$ and $\mathfrak{R}_n$.}
	Notice that in \eqref{convergence pressure terms} we don't consider the case $\gamma=1$ because it reduces to \eqref{convergence densities}. On the other side, when $\gamma>1$, estimates \eqref{estimate kinetic energy} and \eqref{estimate pressure potential}, combined with the fact that
	\begin{equation*}
		\left| \frac{\textbf{m}\otimes \textbf{m}}{\varrho} \right| \leq c \frac{|\textbf{m}|^2}{\varrho}, \quad p(\varrho) \leq c\left(1+ P(\varrho) \right)
	\end{equation*}
	for some positive constants $c$, imply that the pressures $\{ p(\varrho_n(t,\cdot)) \}_{n\in \mathbb{N}}$ and the convective terms $\left\{ \frac{\textbf{m}_n \otimes \textbf{m}_n}{\varrho_n}(t,\cdot) \right\}_{n\in \mathbb{N}}$ are uniformly bounded in the non--reflexive $L^1$--space for a.e. $t>0$. The idea is then to see the $L^1$--space as embedded in the space of bounded Radon measures $\mathcal{M}(\overline{\Omega})$, which in turn can be identified as the dual space of the separable space $C(\overline{\Omega})$. Accordingly, introducing the space $L^{\infty}_{\rm weak}(0,\infty; \mathcal{M}(\overline{\Omega}))$, cf. Remark \ref{remark on space of measures}, we obtain convergences \eqref{convergence convective terms} and \eqref{convergence pressure terms}. Finally, estimate \eqref{estimate defect energy} guarantees convergence \eqref{convergence defects}.
	
	\subsection{Limit passage} \label{Limit passage}
	
	We are now ready to pass to the limit in the weak formulation of the continuity equation and the balance of momentum, obtaining that
	\begin{equation*}
		\left[ \int_{\Omega} \varrho \varphi(t,\cdot) dx \right]_{t=0}^{t=\tau}= \int_{0}^{\tau} \int_{\Omega} [\varrho \partial_t \varphi + \textbf{m}\cdot \nabla_x \varphi] dxdt
	\end{equation*}
	holds $\tau >0$ and any $\varphi \in C_c^1([0,\infty)\times \overline{\Omega})$, with $\varrho(0,\cdot)=\varrho_0$, and
		\begin{align*}
			\left[ \int_{\Omega} \textbf{m}\cdot \bm{\varphi}(t, \cdot) dx \right]_{t=0}^{t=\tau} &= \int_{0}^{\tau}\int_{\Omega} \left[ \textbf{m} \cdot \partial_t \bm{\varphi} +  \overline{\mathbbm{1}_{\varrho>0}\frac{\textbf{m}\otimes \textbf{m}}{\varrho}}:\nabla_x \bm{\varphi} +\overline{p(\varrho)}\divv_x\bm{\varphi}\right] dxdt \\
			&- \int_{0}^{\tau} \int_{\Omega} \mathbb{S}: \nabla_x \bm{\varphi} \ dxdt + \int_{0}^{\tau} \int_{\overline{\Omega}} \nabla_x \bm{\varphi} : d\widetilde{\mathfrak{R}} \ dt
		\end{align*}
		holds for any $\tau>0$ and any $\bm{\varphi} \in C^1_c([0,\infty)\times \overline{\Omega}; \mathbb{R}^d)$, $\bm{\varphi}|_{\partial \Omega}=0$ with $\textbf{m}(0,\cdot)=\textbf{m}_0$. The last integral identity can be rewritten as
		\begin{align*}
			\left[ \int_{\Omega} \textbf{m}\cdot \bm{\varphi}(t, \cdot) dx \right]_{t=0}^{t=\tau} &= \int_{0}^{\tau}\int_{\Omega} \left[ \textbf{m} \cdot \partial_t \bm{\varphi} + \mathbbm{1}_{\varrho>0} \frac{\textbf{m}\otimes \textbf{m}}{\varrho}:\nabla_x \bm{\varphi} +p(\varrho)\divv_x\bm{\varphi}\right] dxdt \\
			&-\int_{0}^{\tau} \int_{\Omega} \mathbb{S}: \nabla_x \bm{\varphi} dxdt + \int_{0}^{\tau} \int_{\overline{\Omega}} \nabla_x \bm{\varphi} : d\check{\mathfrak{R}} \ dt
		\end{align*}
		where $\check{\mathfrak{R}} \in L^{\infty}_{\rm weak}(0,\infty; \mathcal{M}(\overline{\Omega}; \mathbb{R}^{d\times d}_{\rm sym}))$ is such that
		\begin{equation} \label{defect balance of momentum}
			d\check{\mathfrak{R}} = d\widetilde{\mathfrak{R}} +  \left(\overline{\mathbbm{1}_{\varrho>0}\frac{\textbf{m}\otimes \textbf{m}}{\varrho}} -\mathbbm{1}_{\varrho>0}\frac{\textbf{m}\otimes \textbf{m}}{\varrho}\right) dx + \left( \overline{p(\varrho)}-p(\varrho)\right)\chi(\gamma)\mathbb{I} dx,
		\end{equation}
		with
		\begin{equation*}
			\chi(\gamma) =\begin{cases}
				0 &\mbox{if } \gamma=1, \\
				1 &\mbox{if } \gamma>1.
			\end{cases}
		\end{equation*}
		We can prove the stronger condition
		\begin{equation} \label{defect from balance of momentum}
			\check{\mathfrak{R}} \in L^{\infty}_{\rm weak}(0,\infty; \mathcal{M}^+(\overline{\Omega}; \mathbb{R}^{d\times d}_{\rm sym}));
		\end{equation}
		more precisely, we want to show that for all $\xi \in \mathbb{R}^d$, all open sets $\mathcal{B}\subset \Omega$ and a.e. $\tau >0$
		\begin{equation*}
			\check{\mathfrak{R}}(\tau) :(\xi \otimes \xi) (\mathcal{B}) \geq 0.
		\end{equation*}
		we can rewrite the term on the left--hand side as
		\begin{align*}
			\int_{\mathcal{B}} (\xi \otimes \xi): d\check{\mathfrak{R}}(\tau+) &= \lim_{d\rightarrow 0} \int_{\tau}^{\tau+d} \int_{\mathcal{B}} (\xi \otimes \xi): d\check{\mathfrak{R}}(t) dt \\
			&= \lim_{d\rightarrow 0} \int_{0}^{\infty} \int_{\overline{\Omega}} \mathbbm{1}_{[\tau,\tau+d]\times \mathcal{B}} \ (\xi \otimes \xi): d\check{\mathfrak{R}}(t) dt 
		\end{align*}
		Since the indicator function $\mathbbm{1}_{[\tau-d,\tau+d]\times \mathcal{B}}$ can be approximated by some non--negative test functions, it is enough to show that 
		\begin{equation*}
			\int_{0}^{\infty} \int_{\overline{\Omega}} \varphi \ (\xi \otimes \xi): d\check{\mathfrak{R}}(t) dt \geq 0
		\end{equation*}
		holds for all $\varphi \in C_c^{\infty}((0,T)\times \Omega)$, $\varphi\geq 0$. We can notice that the first term on the right--hand side of \eqref{defect balance of momentum} will obviously satisfy the above inequality since $\widetilde{\mathfrak{R}}$ belongs to $L^{\infty}_{\rm weak}(0,\infty; \mathcal{M}^+(\overline{\Omega}; \mathbb{R}^{d\times d}_{\rm sym}))$, and
		\begin{equation*}
			\int_{0}^{\infty} \int_{\overline{\Omega}} \left( \overline{p(\varrho)}-p(\varrho)\right) \mathbb{I} : (\xi \otimes \xi) \ \varphi \ dxdt = \int_{0}^{\infty} \int_{\overline{\Omega}} \left( \overline{p(\varrho)}-p(\varrho)\right) |\xi|^2  \varphi \ dxdt \geq 0,
		\end{equation*}
		since $\varrho \mapsto p(\varrho)$ is convex and weakly lower semi--continuous in $L^1$, which implies $\overline{p(\varrho)} \geq p(\varrho)$, see for instance \cite{Fei}, Theorem 2.11. Finally, following the same idea developed in \cite{FeiHof}, Section 3.2, as a consequence of \eqref{convergence convective terms} we can write
		\begin{align} \label{limit of convective terms}
		\begin{aligned}
			\int_{0}^{\infty} \int_{\overline{\Omega}} \left( \overline{\mathbbm{1}_{\varrho>0}\frac{\textbf{m}\otimes \textbf{m}}{\varrho}} -\mathbbm{1}_{\varrho>0}\frac{\textbf{m}\otimes \textbf{m}}{\varrho} \right)&: (\xi \otimes \xi) \ \varphi \ dxdt \\
			= \lim_{n\rightarrow \infty} \int_{0}^{\infty} \int_{\overline{\Omega}} &\left( \mathbbm{1}_{\varrho_n>0}\frac{\textbf{m}_n \otimes \textbf{m}_n}{\varrho_n} - \mathbbm{1}_{\varrho>0}\frac{\textbf{m}\otimes \textbf{m}}{\varrho} \right): (\xi \otimes \xi) \ \varphi \ dxdt \\
			= \lim_{n\rightarrow \infty} \int_{0}^{\infty} \int_{\overline{\Omega}} &\left( \mathbbm{1}_{\varrho_n>0}\frac{|\textbf{m}_n \cdot \xi|^2}{\varrho_n} - \mathbbm{1}_{\varrho>0}\frac{|\textbf{m}\cdot \xi|^2}{\varrho} \right) \ \varphi \ dxdt.
			\end{aligned}
		\end{align}
		The Cauchy--Schwarz inequality allows to write $|\textbf{m} \cdot \xi|^2 \leq |\textbf{m}|^2 |\xi|^2$, and thus by \eqref{estimate kinetic energy} we obtain
		\begin{equation*}
			\esssup_{t>0} \left\| \mathbbm{1}_{\varrho_n>0}\frac{|\textbf{m}_n \cdot \xi|^2}{\varrho_n} (t,\cdot) \right\|_{L^1(\Omega)} \leq c(\overline{E}, \xi);
		\end{equation*}
		it is possible then to find the limit
		\begin{equation*}
			\mathbbm{1}_{\varrho_n>0}\frac{|\textbf{m}_n \cdot \xi|^2}{\varrho_n} \overset{*}{\rightharpoonup}  \overline{\mathbbm{1}_{\varrho_>0} \frac{|\textbf{m}\cdot \xi|^2}{\varrho}} \quad \mbox{in } L^{\infty}_{\rm weak}(0,\infty; \mathcal{M}(\overline{\Omega}))
		\end{equation*}
		and rewrite the first line in \eqref{limit of convective terms} as 
		\begin{equation*}
			\int_{0}^{\infty} \int_{\overline{\Omega}} \left(   \overline{\mathbbm{1}_{\varrho_>0}\frac{|\textbf{m}\cdot \xi|^2}{\varrho}} - \mathbbm{1}_{\varrho>0}\frac{|\textbf{m}\cdot \xi|^2}{\varrho} \right) \ \varphi \ dxdt.
		\end{equation*}
		as in the previous passage, \eqref{defect from balance of momentum} will now follow from the weak lower semi--continuity on $D$ of the convex function $[\varrho, \textbf{m}] \mapsto \frac{|\textbf{m}\cdot \xi|^2 }{\varrho}$. We proved in particular that the pair of functions $[\varrho, \textbf{m}]$ satisfies conditions (ii) and (iii) of Definition \ref{dissipative solution}. However, $\check{\mathfrak{R}}$ has to be slightly modified in order to get the energy \eqref{energy}, as we will see in Section \ref{Energy convergence}.
		
		\subsection{Auxiliary lemma} \label{Auxuliary lemma}
		
		In order to prove that
		\begin{equation*}
			\textbf{m}=\varrho \textbf{u} \quad \mbox{a.e. in } (0,\infty)\times \Omega,
		\end{equation*}
		and in particular to show that $[\varrho, \textbf{m}]$ satisfy condition (iv) of Definition \ref{dissipative solution}, we need the following result.
		
		\begin{lemma} \label{auxiliary lemma}
			Let $\Omega \subset \mathbb{R}^d$ be a bounded domain. Suppose
			\begin{equation*} 
			\{\varrho_n\}_{n\in \mathbb{N}} \mbox{ to be uniformly bounded in } L^{\infty}(0,\infty; L_{\Phi}(\Omega)),
			\end{equation*}
			where $L_{\Phi}(\Omega)$ is the Orlicz space associated to the Young function $\Phi$ satisfying the $\Delta_2$--condition, with $\varrho_n \geq 0$ for all $n\in \mathbb{N}$. Suppose also that
			\begin{equation} \label{expression of derivative of rho_n}
			\partial_t \varrho_n =-\divv_x \textbf{g}_n
			\end{equation}
			where 
			\begin{equation} \label{uniform integrability of the derivatives of rho_n}
			\{ \textbf{g}_n \}_{n\in \mathbb{N}} \mbox{ is uniformly bounded in } L^{\infty}(0,\infty; L^p(\Omega;\mathbb{R}^d))
			\end{equation}
			for some $p\geq 1$. Moreover, let
			\begin{equation} \label{uniform integraility in W^1,q}
			\{ \textbf{u}_n \}_{n\in \mathbb{N}} \mbox{ be uniformly bounded in } L^q_{\rm loc}(0,\infty; W^{1,q}(\Omega;\mathbb{R}^d)), \ q>1.
			\end{equation}
			Finally, let the sequence $\{\varrho_n \textbf{u}_n\}_{n\in \mathbb{N}}$ be equi--integrable in $L^1_{\rm loc}(0,\infty; L^1( \Omega;\mathbb{R}^d))$.
			
			Then, if
			\begin{align}
			\varrho_n \rightarrow \varrho \quad &\mbox{in } C_{\rm weak, loc}([0,\infty); L^1(\Omega)), \label{convergence of rho_n}\\
			\textbf{u}_n \rightharpoonup \textbf{u} \quad &\mbox{in } L^q_{\rm loc}(0,
			\infty; W^{1,q}(\Omega;\mathbb{R}^d)),
			\end{align}
			and
			\begin{equation*}
			\varrho_n \textbf{u}_n \rightharpoonup \textbf{m} \quad \mbox{in } L^1_{\rm loc}(0,\infty; L^1( \Omega; \mathbb{R}^d)),
			\end{equation*}
			we have
			\begin{equation*}
			\textbf{m}=\varrho \textbf{u} \quad \mbox{a.e. in }(0,\infty)\times \Omega.
			\end{equation*}
		\end{lemma}
		
		\begin{proof}
			\begin{enumerate}
				\item \textbf{\textit{Truncation.}}
				Following the same strategy developed in \cite{AbbFei}, Lemma 8.1, it is enough to suppose that 
				\begin{equation} \label{uniform integrability in Linfinity}
				\{ \textbf{u}_n \}_{n\in \mathbb{N}} \mbox{ is uniformly bounded in } L^{\infty}_{\rm loc}(0,\infty; L^{\infty}( \Omega; \mathbb{R}^d)).
				\end{equation}
				Considering the truncation $T_k(\textbf{u}_n)$, notice that in this case we have
				\begin{equation*}
				\int_{a}^{b} \int_{\Omega} \varrho_n |\textbf{u}_n - T_k(\textbf{u}_n)| \ dxdt \leq 2 \int \int_{\{|\textbf{u}_n| \geq k\}} \varrho_n |\textbf{u}_n| \ dxdt \rightarrow 0 \quad \mbox{as } k\rightarrow \infty, \mbox{ uniformly in }n,
				\end{equation*}
				for every $[a,b]\subset (0,\infty)$, in view of the equi--integrability of $\{ \varrho_n \textbf{u}_n \}_{n\in \mathbb{N}}$.
				\item \textbf{\textit{Regularization.}} We claim that it is sufficient to suppose
				\begin{equation} \label{relaxation of requirement on u}
				\{ \textbf{u}_n \}_{n\in \mathbb{N}} \mbox{ to be uniformly bounded in } L^q_{\rm loc}(0,\infty; W^{m,r}(\Omega;\mathbb{R}^d))
				\end{equation}
				with $q>1$ and $m,r$ arbitrarily large. Seeing all the quantities involved as embedded in $\mathbb{R}^d$ with compact support, we consider regularization in the spatial variable by convolution with a family of regularizing kernels $\{ \theta_{\delta} \}_{\delta>0}$,
				\begin{equation*}
				\theta_{\delta}(x)= \frac{1}{\delta^d} \ \theta \left( \frac{x}{\delta} \right),
				\end{equation*}
				where $\theta$ is a bell--shaped function such that
				\begin{equation*}
				\theta \in C^{\infty}_c (\mathbb{R}^d), \quad \theta \geq 0, \quad \theta(x)=\theta(|x|), \quad \int_{\mathbb{R}^d} \theta(x) dx=1.
				\end{equation*}
				As in the previous step, writing 
				\begin{equation*}
				\varrho_n \textbf{u}_n = \varrho_n \ \theta_{\delta} * \textbf{u}_n + \varrho_n (\textbf{u}_n - \theta_{\delta} * \textbf{u}_n),
				\end{equation*}
				our goal is tho show that for every $[a,b]\subset (0,\infty)$
				\begin{equation*}
				\int_{a}^{b} \int_{\Omega} \varrho_n |\textbf{u}_n - \theta_{\delta} * \textbf{u}_n | dxdt \rightarrow 0 \quad \mbox{as } \delta\rightarrow 0, \mbox{ uniformly in }n.
				\end{equation*}
				
				To this end, we introduce the Banach space 
				\begin{equation*}
				X= W_0^{1,q} \cap L^{\infty} (\Omega;\mathbb{R}^d)
				\end{equation*}
				and observe that, in view of \eqref{uniform integraility in W^1,q} and \eqref{uniform integrability in Linfinity},
				\begin{equation} \label{norms in X belong to Lq}
				\{ \|\textbf{u}_n \|_X \}_{n\in \mathbb{N}} \mbox{ is uniformly bounded in } L^q_{\rm loc}(0,\infty).
				\end{equation}
				
				Consequently, we may write 
				\begin{equation*}
				\int_{a}^{b} \int_{\Omega} \varrho_n |\textbf{u}_n - \theta_{\delta} * \textbf{u}_n | dxdt = I_1^M + I_2^M
				\end{equation*}
				with
				\begin{align*}
				I_1^M &= \int_{\{ \|\textbf{u}_n(t, \cdot)\|_X \leq M \}} \int_{\Omega} \varrho_n |\textbf{u}_n - \theta_{\delta} * \textbf{u}_n | dxdt, \\
				I_2^M &= \int_{\{ \|\textbf{u}_n(t, \cdot)\|_X > M \}} \int_{\Omega} \varrho_n |\textbf{u}_n - \theta_{\delta} * \textbf{u}_n | dxdt,
				\end{align*}
				where, in view of \eqref{norms in X belong to Lq} - recall that the functions $\varrho_n$ are weakly continuous in time
				\begin{equation*}
				I_2^M \leq c \sup_{t\in [a,b]}\|\varrho_n(t,\cdot)\|_{L^1(\Omega)} \|\textbf{u}_n\|_{L^{\infty}((a,b)\times \Omega; \mathbb{R}^d)} \ |\{ \|\textbf{u}_n(t, \cdot)\|_X > M \}| \rightarrow 0 \quad \mbox{as } M\rightarrow \infty,
				\end{equation*}
				uniformly in $n$ and independently of $\delta$.
				
				It remains to show smallness of the first integral for fixed $M$. To this end, denoting with $\Psi$ the complementary Young function of $\Phi$, we consider the Orlicz space $L_{\Psi}(\Omega)$ that can be identified with the dual of $L_{\Phi}(\Omega)$ as $\Phi$ satisfies the $\Delta_2$--condition. By Proposition \ref{compact embedding in Orlicz space} below, we recover the compact embedding 
				\begin{equation*}
				X \hookrightarrow\hookrightarrow L_{\Psi}(\Omega;\mathbb{R}^d)
				\end{equation*}
				which, combined with boundedness of convolution on $L_{\Psi}(\Omega)$ (see \cite{HarHas}, Lemma 4.4.3), gives 
				\begin{equation*}
				I_1^M \leq \sup_{t\in [a,b]} \| \varrho_n(t,\cdot)\|_{ L_{\Phi}(\Omega)} \sup_{\|\textbf{u}\|_X\leq M} \|\textbf{u} -\theta_{\delta} * \textbf{u}\|_{L_{\Psi}(\Omega;\mathbb{R}^d)}\rightarrow 0 \quad \mbox{as } \delta\rightarrow 0.
				\end{equation*}
				\item \textbf{\textit{Conclusion.}} Using the Sobolev embedding
				\begin{equation*}
					L^1(\Omega) \hookrightarrow\hookrightarrow W^{-1, s'} \quad \mbox{for any } s>d,
				\end{equation*}
				from \eqref{convergence of rho_n} we get that
				\begin{equation*}
				\varrho_n \rightarrow \varrho \quad \mbox{in } C_{\rm loc}([0,\infty); W^{-1, s'}(\Omega)) \quad \mbox{for any } s>d,
				\end{equation*}
				and thus, to conclude the proof of the Lemma it is sufficient to choose $m=1$ and $r=s$ in \eqref{relaxation of requirement on u}. 
			\end{enumerate}
		\end{proof}
		
		\begin{proposition} \label{compact embedding in Orlicz space}
			Let $\Omega\subset \mathbb{R}^d$ be a bounded domain. Then, for a fixed $q\geq1$
			\begin{equation*}
			X=W^{1,q}_0\cap L^{\infty}(\Omega) \hookrightarrow\hookrightarrow L_{\Phi}(\Omega),
			\end{equation*}
			where $L_{\Phi}$ is the Orlicz space associated to the Young function $\Phi$.
		\end{proposition}
		\begin{proof}
			Let $K$ be a bounded set of $X$ and let $\Phi_1$ be a Young function such that $\Phi \prec \prec \Phi_1$, i.e. 
			\begin{equation*}
			\lim_{t\rightarrow \infty} \frac{\Phi(t)}{\Phi_1(\lambda t)}=0,
			\end{equation*}
			for all $\lambda >0$. Then, in particular, $K$ is bounded in the Orlicz space $L_{\Phi_1}(\Omega)$; indeed, denoting with $\Psi$ the complementary Young function of $\Phi$, we have that for every $u\in K$ and every $v$ belonging to the Orlicz class $\widetilde{L}_{\Psi}(\Omega)$
			\begin{equation*}
			\int_{\Omega}|u(x)v(x)| dx \leq \|u\|_{L^{\infty}(\Omega)} \|v\|_{L^1(\Omega)} \leq \|u\|_{L^{\infty}(\Omega)} \ \sigma(v;\Psi),
			\end{equation*}
			and thus
			\begin{equation*}
				\|u\|_{L_{\Phi_1}(\Omega)} =\sup_{\substack{v\in \widetilde{L}_{\Psi}(\Omega) \\ \sigma(v;\Psi)\leq1}} \int_{\Omega}|u(x)v(x)| dx \leq \|u\|_{L^{\infty}(\Omega)} \leq \|u\|_X \leq c,
			\end{equation*}
			where $\|u\|_X =\max\{ \|u\|_{W^{1,q}(\Omega)}, \|u\|_{L^{\infty}(\Omega)} \}$ and the constant $c$ is independent of the choice $u\in K$. Applying \cite{KufJohFuc}, Theorems 3.17.7 and 3.17.8 we get that
			\begin{equation*}
			L_{\Phi_1}(\Omega)\hookrightarrow E_{\Phi}(\Omega),
			\end{equation*}
			where $E_{\Phi}(\Omega)$ is the closure of the set of all bounded measurable functions defined on $\Omega$ with respect to the Orlicz norm $\| \cdot \|_{L_{\Phi}}$, and that the functions in $K$ have uniformly continuous $L_{\Phi}$--norms, i.e., for every $\varepsilon>0$ there exists a $\delta=\delta(\varepsilon)>0$ such that
			\begin{equation*}
			\|u \mathbbm{1}_M\|_{L_{\Phi}(\Omega)} <\varepsilon,
			\end{equation*} 
			provided $M\in \Omega$ is measurable, $|M|<\delta$ and $u\in K$.
			
			Furthermore, since 
			\begin{equation*}
			W^{1,q}(\Omega) \hookrightarrow\hookrightarrow L^1(\Omega),
			\end{equation*}
			the set $K$ is relatively compact in $L^1(\Omega)$ and consequently it is relatively compact with respect to the convergence in measure.
			
			Finally, it is sufficient to apply \cite{KufJohFuc}, Theorem 3.14.11, which we report for reader's convenience.
			\begin{theorem}
				Let $K$ be a subset of $E_{\Phi}(\Omega)$ which is relatively compact in the sense of convergence in measure and such that the functions in $K$ have uniformly continuous $L_{\Phi}$--norms. Then $K$ is relatively compact in $L_{\Phi}$.
			\end{theorem}
		\end{proof}
		
		\subsection{Limit of the energies} \label{Energy convergence}
		From \eqref{energy} we can notice that the energies $E_n(\tau)$ are non--increasing and for $\gamma>1$ they are also non--negative, while for $\gamma=1$ we have
		\begin{align*}
			E_n(\tau) &\geq \int_{\Omega} \left[ \frac{1}{2}  \frac{|\textbf{m}_n|^2}{\varrho_n} + \mathbbm{1}_{\varrho_n \geq 1} \ \varrho\log \varrho\right] (\tau, \cdot) dx +\frac{1}{\lambda_n}\int_{\overline{\Omega}} d \trace[\mathfrak{R}_n (\tau)] + \int_{\{ 0\leq \varrho_n < 1 \}} \varrho_n \log \varrho_n dx \\
			&\geq \mbox{``non-negative term"} -\frac{|\Omega|}{e}.
		\end{align*}
		for a.e. $\tau >0$. Hence, for every $[a,b]\subset (0,\infty)$ and every $n\in \mathbb{N}$
		\begin{equation*}
		\|E_n\|_{L^1[a,b]}  \leq  \int_{0}^{b} |E_n(t)| dt \leq b \ \sup_{t\in[0,b]} |E_n(t)| \leq b \ E_{0,n} \leq c(\overline{E}),
		\end{equation*}
		\begin{equation*}
			V_a^b(E_n) =\int_{a}^{b} \left|\frac{d}{dt} E_n \right|= -\int_{a}^{b} \frac{d}{dt} E_n dt = E_n(a) - E_n(b) \leq \begin{cases}
				E_{0,n} + \frac{|\Omega|}{e} &\mbox{if } \gamma=1 \\
				E_{0,n}  &\mbox{if } \gamma>1
			\end{cases} \leq c(\overline{E}),
		\end{equation*}
		so that $\{E_n\}_{n\in \mathbb{N}}$ is locally of bounded variation. We can then use Helly's selection theorem (compactness theorem for $BV_{\rm loc}$): a sequence of functions that is locally of total bounded variation and uniformly bounded at a point has a convergent subsequence, pointwise and in $L^1_{\rm loc}$. Passing to a suitable subsequence as the case may be, we obtain
		\begin{equation} \label{energies convergence}
			E_n(t) \rightarrow E(t) \quad \mbox{for every } t\in [0,\infty) \mbox{ and in } L^1_{\rm loc}(0, \infty),
		\end{equation}
		which in particular implies
		\begin{equation} \label{convergence energies}
			E_n \rightarrow E \quad \mbox{in } \mathfrak{D}([0,\infty); \mathbb{R}),
		\end{equation}
		since $E_n: [0,\infty) \rightarrow \mathbb{R}$ is a monotone function for all $n\in \mathbb{N}$ and thus, by Proposition \ref{characterization convergence non-compact hilbert}, showing \eqref{convergence energies} is equivalent to show almost everywhere convergence.
		
		On the other side, from \eqref{estimate kinetic energy}, \eqref{estimate pressure potential} and \eqref{estimate defect energy} we get
		\begin{align*}
		\frac{|\textbf{m}_n|^2}{\varrho_n} \overset{*}{\rightharpoonup} \overline{\frac{|\textbf{m}|^2}{\varrho}} \quad &\mbox{in } L^{\infty}_{\rm weak}(0,\infty; \mathcal{M}(\overline{\Omega})) \\
		P(\varrho_n) \overset{*}{\rightharpoonup} \overline{P(\varrho)} \quad &\mbox{in } L^{\infty}_{\rm weak}(0,\infty; \mathcal{M}(\overline{\Omega})) \\
		\frac{1}{\lambda_n}\trace \left[\mathfrak{R}_n\right] \overset{*}{\rightharpoonup} \widetilde{\mathfrak{E}} \quad &\mbox{in } L^{\infty}_{\rm weak}(0,\infty; \mathcal{M}^+(\overline{\Omega})).
		\end{align*}
		We can then write 
		\begin{equation} \label{problem1}
		E(\tau) = \int_{\Omega} \left[ \frac{1}{2} \frac{|\textbf{m}|^2}{\varrho} + P(\varrho) \right](\tau,\cdot) dx + \int_{\overline{\Omega}} d \mathfrak{E}(\tau)
		\end{equation}
		for a.e. $\tau>0$, with
		\begin{equation*}
		d\mathfrak{E} = d \widetilde{\mathfrak{E}} + \frac{1}{2}\left( \overline{\frac{|\textbf{m}|^2}{\varrho}} - \frac{|\textbf{m}|^2}{\varrho} \right) dx + \left( \overline{P(\varrho)}- P(\varrho) \right) dx
		\end{equation*}
		where, once again, from the convexity of the function $P$ and of the superposition $[\varrho, \textbf{m}]\mapsto \frac{|\textbf{m}|^2}{\varrho}$, we get
		\begin{equation*}
			\mathfrak{E} \in L^{\infty}_{\rm weak}(0,\infty; \mathcal{M}^+(\overline{\Omega})).
		\end{equation*}
		As pointed out in Section \ref{remark reynold stress}, we can choose  constant $\lambda>0$ such that
		\begin{equation} \label{problem 2}
			\trace [\check{\mathfrak{R}}(\tau)] \leq \lambda \mathfrak{E}(\tau)
		\end{equation}
		for a.e. $\tau \in (0,T)$; however, with this choice we only get
		\begin{equation*}
			\int_{\Omega} \left[ \frac{1}{2} \frac{|\textbf{m}|^2}{\varrho} + P(\varrho) \right](\tau,\cdot) \ {\rm d}x + \frac{1}{\lambda}\int_{\overline{\Omega}} \rm d  \trace[\check{\mathfrak{R}}(\tau)] \leq E(\tau)
		\end{equation*}
		for a.e. $\tau \in (0,T)$. To obtain \eqref{energy}, it is sufficient to define a new defect
		\begin{equation*}
			\mathfrak{R}= \check{\mathfrak{R}} + \psi(t)\mathbb{I},
		\end{equation*}
		where the function $\psi\geq 0$ of time only can be chosen in such a way that
		\begin{equation*}
			\int_{\Omega} \left[ \frac{1}{2} \frac{|\textbf{m}|^2}{\varrho} + P(\varrho) \right](\tau,\cdot) \ {\rm d}x + \frac{1}{\lambda}\int_{\overline{\Omega}} \rm d  \trace[\mathfrak{R}(\tau)] = E(\tau)
		\end{equation*}
		for a.e. $\tau \in (0,T)$. Clearly,
		\begin{equation*}
			\int_{\overline{\Omega}} \nabla_x \bm{\varphi} : d\mathfrak{R} = \int_{\overline{\Omega}} \nabla_x \bm{\varphi} : d\check{\mathfrak{R}}
		\end{equation*}
		for any $\bm{\varphi} \in C_c^{\infty}([0,\infty) \times \overline{\Omega}; \mathbb{R}^d)$, $\bm{\varphi}|_{\partial \Omega}=0$, and therefore, the weak formulation of the balance of momentum \eqref{weak formulation balance of momentum} remains valid.
		
		Finally, notice that the couple $[\varrho, \textbf{m}]$ satisfies the energy inequality \eqref{energy inequality} due to lower semi--continuity of the functions $F$ and $F^*$: for a.e. $\tau>0$
		\begin{equation*}
			\int_{0}^{\tau} \int_{\Omega} [F(\mathbb{D}\textbf{u}) +F^*(\mathbb{S})] dxdt \leq \liminf_{n\rightarrow \infty}\int_{0}^{\tau} \int_{\Omega} [F(\mathbb{D}\textbf{u}_n) +F^*(\mathbb{S}_n)] dxdt;
		\end{equation*} 
		in particular, $[\varrho, \textbf{m}]$ satisfies condition (iv) of Definition \ref{dissipative solution}.

	\bigskip
	
	\centerline{\bf Acknowledgement}
	
	This work was supported by the Einstein Foundation, Berlin. The author wishes to thank Prof. Eduard Feireisl for the helpful advice and discussions.

\end{document}